\documentclass[12pt]{article}

\usepackage{amsthm}
\usepackage{amsfonts}
\usepackage{amsmath}
\usepackage{amssymb}
\usepackage{color}

\usepackage{pictex}
\usepackage{multicol,url}

\begin{document}
 \newcounter{thlistctr}
 \newenvironment{thlist}{\
 \begin{list}%
 {\alph{thlistctr}}%
 {\setlength{\labelwidth}{2ex}%
 \setlength{\labelsep}{1ex}%
 \setlength{\leftmargin}{6ex}%
 \renewcommand{\makelabel}[1]{\makebox[\labelwidth][r]{\rm (##1)}}%
 \usecounter{thlistctr}}}%
 {\end{list}}

\thispagestyle{empty}

\newtheorem{Lemma}{\bf LEMMA}[section]
\newtheorem{Theorem}[Lemma]{\bf THEOREM}
\newtheorem{Claim}[Lemma]{\bf CLAIM}
\newtheorem{Corollary}[Lemma]{\bf COROLLARY}
\newtheorem{Proposition}[Lemma]{\bf PROPOSITION}
\newtheorem{Example}[Lemma]{\bf EXAMPLE}
\newtheorem{Fact}[Lemma]{\bf FACT}
\newtheorem{definition}[Lemma]{\bf DEFINITION}
\newtheorem{Notation}[Lemma]{\bf NOTATION}
\newtheorem{remark}[Lemma]{\bf REMARK}

\title{Symmetric Implication Zroupoids and Weak Associative Laws}
             
\author{Juan M. Cornejo and Hanamantagouda P. Sankappanavar} 

\date{}

\maketitle

\numberwithin{equation}{section}

\thispagestyle{empty}

\begin{abstract}
 An algebra $\mathbf A = \langle A, \to, 0 \rangle$, where $\to$ is binary and $0$ is a constant, is called 
	 an {\it implication zroupoid} ($\mathcal I$-zroupoid, for short) if $\mathbf A$ satisfies the identities:	
	 	
	 	$(x \to y) \to z \approx ((z' \to x) \to (y \to z)')'$
	 and   $ 0'' \approx 0$, where $x' : = x \to 0$. 
	 An implication zroupoid is {\it symmetric} if it satisfies $x'' \approx x$ and $(x \to y')' \approx (y \to x')'$.  The variety of symmetric $\mathcal I$-zroupoids is denoted by $\mathcal S$.  
	We began a systematic analysis of weak associative laws of  length $\leq 4$ in \cite{cornejo2016BolMoufang}, by examining the identities of Bol-Moufang type in the context of the variety 
$\mathcal S$.

In this paper we complete the analysis by investigating the rest of the   
weak associative laws of length $\leq 4$ relative to $\mathcal S$.  We show that, of the 155 
subvarieties of $\mathcal S$ defined by the weak associative laws of size $\leq 4$,  
there are exactly $6$ distinct ones.
We also give an explicit description of the poset of the (distinct) subvarieties of $\mathcal S$ defined by weak associative laws of length $\leq 4$. 
\end{abstract}

\section{{\bf Introduction}} \label{SA}

In 1934, Bernstein gave a system of axioms for Boolean algebras in \cite{bernstein1934postulates} using implication alone.  
Even though his system was not equational, it is not hard to see that one could easily convert it into an equational one by using an additional constant.  
In 2012, the second author extended this ``modified Bernstein's theorem''
to De Morgan algebras in \cite{sankappanavarMorgan2012}.  
Indeed, he shows in \cite{sankappanavarMorgan2012} that the varieties of De Morgan algebras, Kleene algebras, and  Boolean algebras are term-equivalent, respectively, to the varieties, $\mathcal{DM}$, $\mathcal{KL}$, and $\mathcal{BA}$ (defined below)  
whose defining axioms use only an implication $\to$ and a constant $0$.    

The 
primary role played by the identity (I): $(x \to y) \to z \approx ((z' \to x) \to (y \to z)')'$, where $x' : = x \to 0$, 
in the axiomatization of each of those new varieties 
motivated the second author to 
introduce a new equational class of algebras called ``Implication zroupoids'' in   
\cite{sankappanavarMorgan2012}.  It also turns out that this new variety contains the variety $\mathcal{SL}$ (defined below) which is term-equivalent to the variety of $\lor$-semilattices with the least element $0$ (see \cite{cornejo2015implication}).

%

\begin{definition}
	An algebra $\mathbf A = \langle A, \to, 0 \rangle$, where $\to$ is binary and $0$ is a constant, is called a {\rm zroupoid}.
	A zroupoid $\mathbf A = \langle A, \to, 0 \rangle$ is an {\rm Implication zroupoid} \rm ($\mathcal I$-zroupoid, for short\rm) if $\mathbf A$ satisfies:
	\begin{itemize}
		\item[\rm{(I)}] 	$(x \to y) \to z \approx ((z' \to x) \to (y \to z)')'$, where $x' : = x \to 0$,
		\item[{\rm (I$_{0}$)}]  $ 0'' \approx 0$.
	\end{itemize}
$\mathcal{I}$ denotes the variety of implication zroupoids.
The varieties $\mathcal{DM}$ 
and $\mathcal{SL}$ 
are defined relative to $\mathcal{I}$, respectively, by the following identities: 
$$
\begin{array}{ll}
{\rm (DM)} & (x \to y) \to x \approx x \mbox{ (De Morgan Algebras)}; \\
{\rm (SL)}  & x' \approx x \mbox{ and } x \to y \approx y \to x.  
\end{array}
$$
The varieties 
$\mathcal{KL}$ and $\mathcal{BA}$ 
are defined relative to $\mathcal{DM}$, respectively, by the following identities: 
$$
\begin{array}{ll}
{\rm (KL)} & (x \to x) \to (y \to y) \approx y \to y  \mbox{ (Kleene algebras)} ; \\
{\rm (BA)} & x \to x \approx 0'  \mbox{ (Boolean algebras)}.\\
\end{array}
$$ 
\end{definition}

As proved in \cite{sankappanavarMorgan2012}, the variety $\mathcal{I}$ 
generalizes the variety of De Morgan algebras 
and exhibits several interesting properties; 
for example, 
the identity $x''' \to y \approx x' \to y$ holds in $\mathcal{I}$.  
Several new subvarieties of $\mathcal{I}$ are also introduced and investigated in
\cite{sankappanavarMorgan2012}.  
The (still largely unexplored) lattice of subvarieties of $\mathcal{I}$ seems to be fairly complex.  In fact, Problem 6 of \cite{sankappanavarMorgan2012} calls for an investigation of the structure of the lattice of subvarieties of $\mathcal{I}$.

The papers  \cite{cornejo2015implication}, \cite{cornejo2016order} and \cite{cornejo2016semisimple} have addressed further, but still partially, the above-mentioned problem by introducing new subvarieties of $\mathcal{I}$ and investigating relationships among them.  
The (currently known) size of the poset of subvarieties of $\mathcal{I}$ is at least 24; but it is still unknown whether the lattice of subvarieties is finite or infinite.  Two of the subvarieties of $\mathcal{I}$ are:
 $\mathcal{I}_{2,0}$ and $\mathcal{MC}$ which are defined relative to $\mathcal{I}$, respectively, by the following identities, where $x \land y := (x \to y')' $ :\\     
(I$_{2,0}) \quad  x'' \approx x$;\\
(MC)\  $x \land y \approx y \land x$. \\
For a somewhat more detailed summary of the results contained in the above-mentioned papers, we refer the reader to the Intoduction of \cite{cornejo2016derived}.

\begin{definition}
	Let $\mathbf{A} \in \mathcal{I}$.  $\mathbf{A}$ is {\rm involutive} if $\mathbf{A} \in \mathcal{I}_{2,0}$.  $\mathbf{A}$ is {\rm meet-commutative} if $\mathbf{A} \in \mathcal{MC}$.  $\mathbf{A}$ is {\rm symmetric} if $\mathbf{A}$ is both involutive and meet-commutative.  Let $\mathcal{S}$ denote the variety of symmetric $\mathcal{I}$-zroupoids.  In other words, $\mathcal{S} = \mathcal{I}_{2,0} \cap \mathcal{MC}$.
\end{definition}

In the present paper we are interested in the subvarieties of $\mathcal{S}$ defined by certain {\it weak associative laws}. 
A precise definition of a {\it weak associative law} appears in \cite{Ku96}, which is essentially restated below.

\begin{definition}
	Let $n \in \mathbb{N}$ and let $\mathcal{L} := \langle \times \rangle$, where $\times$ is a binary operation symbol.  A (groupoid) term in the language $\mathcal{L}$ 
is of length $n$ if the number of occurrences of variables (not necessarily distinct) is $n$.   
A weak associative law of length $n$ in $\mathcal{L}$   
is an identity of the form $p \approx q$, where $p$ and $q$ are terms of length $n$ and contain the same variables that occur in the same order in both 
	$p$ and $q$ \rm (only the bracketing is possibly different\rm).     
	
\end{definition}




The following (general) problem was raised in \cite{cornejo2016BolMoufang} (In the sequel we use the words ``law'' and ``identity'' interchangeably). \\ 

\noindent PROBLEM:  Let $\mathcal V$ be a given variety of algebras (whose language includes a  binary operation symbol, say, $\times$).  
Investigate the subvarieties of $\mathcal V$ defined by   
weak associative laws {\rm (}with respect to $\times${\rm)} 
and their mutual relationships.\\ 


 Special cases of the above problem have already 
been considered in the literature, wherein the weak associative laws chosen are the identities of Bol-Moufang type (i.e., weak associative laws of length $4$ with $3$ distinct variables),
and the variety ${\mathcal V}$ is chosen to be the variety of quasigroups or the variety of loops
(For more information about these identities in the context of quasigroups and loops see \cite{Fe69}, \cite{Ku96}, \cite{PV05a}, \cite{PV05b}). 

Let $W$ denote the set of identities of weak associative laws of size $\leq 4$.   The systematic notation given in the next definition for the identities in $W$ is influenced by the notation developed in \cite{PV05a} for Bol-Moufang identities. 

 Without loss of generality, we will assume that the variables in the terms $t_1$ and $t_2$ occur 
alphabetically in any weak associative identity $t_1 \approx t_2$.  Given a word $X$ of variables, we refer to each possible way of bracketting X that will transform $X$ into a term, as ``bracketting''; and we assign a number to each such bracketting and call it the ``bracketting number'' of that term. 

\begin{definition} 
Let $n,m,p,q \in \mathbb N$ and let $X$ denote a word of length $n$ in which the (not necessarily distinct) $m$ variables occur alphabetically.  
We denote by {\rm (nmXpq)} the weak associative identity $t_1 \approx t_2$ of length $n$ with $m$ distinct variables, where $t_1$ and $t_2$ are terms obtained from $X$ whose bracketing numbers are $p$ and $q$.       
We denote by $nm\mathcal X pq$ the variety defined, relative to $\mathcal S$, by the weak associative identity {\rm (nmXpq)}.  
\end{definition}

\begin{Example} Let $X$ be the word  \rm A:= $\langle x, x, x  \rangle$.  Then there are two possible brackettings (numbered $1$ and $2$): 

\rm1: $a \to (a \to a)$, and
\rm 2: $(a \to a) \to a$. \\
\  \\
Then (31A12) denotes the identity $x \to (x \to x) \approx (x \to x) \to x$ of length 3 with one variable, and with bracketting numbers 1 and 2.  Also, $31\mathcal A12$ denotes the subvariety of $\mathcal I $ defined by (31A12).  Here is another example: if $X$ is the word B: $\langle x, y, x, z \rangle$, then (43B23) denotes the identity $x \to ((y \to x) \to z)$ $\approx$ $(x \to y) \to (x \to z)$ of length 4 with 3 distinct variables and with bracketting numbers 2 and 3 in the listing of possible brackettings given in Section \ref{section_length4_3variables} below.   
\end{Example}
 


Recall that Bol-Moufang identities are precisely the weak associative laws of the form (43Xpq).
It was shown in \cite{cornejo2016BolMoufang} that 
 there are 4 nontrivial subvarieties of $\mathcal S$ of Bol-Moufang type that are distinct from each other (out of $60$ varieties), giving the Hasse diagram of the poset formed by them, together with the variety $\mathcal{BA}$ (which is contained in some of them).

In this paper we will complete that analysis by examining the rest of the 
weak associative laws of size $\leq 4$ relative to $\mathcal S$.  Clearly, $x \approx x$ and $x \to y \approx x \to y$ are the only identitities of length 1 and 2 respectively, which are trivial.  So, we will consider the identities of length 3 and 4.  
In Section \ref{section_main_theorem}, we show, as our main result, that, of the 155 
subvarieties of $\mathcal S$ defined by the weak associative laws of size $\leq 4$,  
there are exactly $6$ distinct ones.
We also give an explicit description, by a Hasse diagram, of the poset of the (distinct) subvarieties of $\mathcal S$ defined by weak associative laws of length $\leq 4$. 

We would like to acknowledge that the software ``Prover 9/Mace 4'' developed by McCune \cite{Mc} has been useful to us in some of our findings presented in this paper.  We have used it to find examples and to check some conjectures.\\

\section{Preliminaries and Properties  of $\mathcal S$}

\indent We refer the reader to the standard references
\cite {balbesDistributive1974}, \cite{burrisCourse1981} and \cite{Ra74} for concepts and results used, but not explained, in this paper.

Recall from \cite{sankappanavarMorgan2012} that $\mathcal{SL}$ is the variety of semilattices with a least element $0$. It was shown in \cite{cornejo2015implication} that $\mathcal{SL}   
= \mathcal{C} \cap \mathcal{I}_{1,0} $.

The two-element algebras $\mathbf{2_s}$, $\mathbf{2_b}$ were introduced in \cite{sankappanavarMorgan2012}.  Their  
operations $\to$ are respectively as follows:  \\

%
\begin{minipage}{0.3 \textwidth}
	\begin{tabular}{r|rr}
		$\to$: & 0 & 1\\
		\hline
		0 & 0 & 1 \\
		1 & 1 & 1
	\end{tabular} 
	
\end{minipage}
\begin{minipage}{0.3 \textwidth}
	\begin{tabular}{r|rr}
		$\to$: & 0 & 1\\
		\hline
		0 & 1 & 1 \\
		1 & 0 & 1
	\end{tabular}   
\end{minipage}           
\ \\ \ \\ \ \\
Recall that $\mathcal{V}(\mathbf{2_b}) = \mathcal{BA}$.
Recall also from \cite[Corollary 10.4]{cornejo2015implication} \label{CorSL} that
$\mathcal{V}(\mathbf{2_s}) = \mathcal{SL}$.  The following lemma easily follows from the definition of $\land$ given earlier in the Introduction.

\begin{Lemma}  \bf \cite{cornejo2015implication} \label{lemma_SL_I10_C}
	$\mathcal{MC} \cap \mathcal{I}_{1,0} \subseteq  \mathcal{C} \cap \mathcal{I}_{1,0} = \mathcal{SL}$.  
\end{Lemma}

\begin{Lemma} {\bf \cite[Theorem 8.15]{sankappanavarMorgan2012}} \label{general_properties_equiv}
	Let $\mathbf A$ be an  $\mathcal{I}$-zroupoid. Then the following are equivalent:
	\begin{enumerate}
		\item $0' \to x \approx x$, \label{TXX} 
		\item $x'' \approx x$,
		\item $(x \to x')' \approx x$, \label{reflexivity}
		\item $x' \to x \approx x$. \label{LeftImplicationwithtilde}
	\end{enumerate}
\end{Lemma}

Recall that $\mathcal{I}_{2,0}$ and $\mathcal{MC}$ are the subvarieties defined, respectivaly, relative to $\mathcal I$ by the equations
\begin{equation} \label{eq_I20} \tag{$I_{2,0}$}
x'' \approx x.
\end{equation}
\begin{equation} \label{eq_MC} \tag{MC}
x \wedge y \approx y \wedge x.
\end{equation}

\begin{Lemma}{\bf \cite{sankappanavarMorgan2012}} \label{general_properties}
	Let $\mathbf A \in  \mathcal{I}_{2,0}$. Then  
	\begin{enumerate}
		\item $x' \to 0' \approx 0 \to x$, \label{cuasiConmutativeOfImplic2}
		\item $0 \to x' \approx x \to 0'$. \label{cuasiConmutativeOfImplic}
	\end{enumerate}
\end{Lemma}

\begin{Lemma} \label{general_properties3}  \label{general_properties2}
	Let $\mathbf A \in \mathcal{I}_{2,0}$. Then $\mathbf A$ satisfies:
	\begin{enumerate}
		\item $(x \to 0') \to y \approx (x \to y') \to y$, \label{281014_05}
		\item $((0 \to x) \to y) \to x \approx y \to x$, \label{291014_08}
		\item $(x \to (y \to x)')' \approx (x \to y) \to x$, \label{291014_09}
		\item $(y \to x) \to y \approx (0 \to x) \to y$, \label{291014_10}
		\item $(0 \to x) \to (x \to y) \approx x \to (x \to y)$, \label{311014_05}
		\item $(0 \to x) \to (0 \to y) \approx x \to (0 \to y)$, \label{311014_06}
		\item $x \to y \approx x \to (x \to y)$, \label{031114_04} 
		\item $0 \to (0 \to x)' \approx 0 \to x'$, \label{031114_07}
		\item $0 \to (x \to y) \approx x \to (0 \to y)$, \label{071114_04}
		\item $0 \to (x \to y')' \approx 0 \to (x' \to y)$, \label{191114_05}
		\item $x \to (y \to x') \approx y \to x'$, \label{281114_01}
		\item $(x \to y) \to (y \to x) \approx y \to x$,  \label{250315_05}
		
\item $(x \to y) \to (y \to z) \approx (0 \to x') \to (y \to z)$, \label{250615_04} 
\item $(x \to y)' \to y \approx x \to y$, \label{250615_01}  
\item $(x \to y) \to ((0 \to y) \to z) \approx (x \to y) \to z$,  \label{260615_07} 
\item $(x \to y) \to ((z \to y) \to (u \to z)) \approx  (x \to y) \to (u \to z).$ \label{070715_06}
	\end{enumerate}
\end{Lemma}

\begin{proof}	
For the proofs of items (\ref{281014_05}), (\ref{291014_09}), (\ref{291014_10}), (\ref{071114_04}), (\ref{191114_05}), (\ref{281114_01}) we refer the reader to \cite{cornejo2015implication}, and
for the proofs of items (\ref{291014_08}), (\ref{311014_06}), (\ref{031114_04}),   (\ref{031114_07}) to \cite{cornejo2016order}.  Items  (\ref{311014_05}), (\ref{250315_05}) are proved in \cite{cornejo2016semisimple}.
For the proofs of items (\ref{250615_04}), (\ref{250615_01}) we refer the reader to \cite{cornejo2016derived}.
Finally, for the proof of (\ref{260615_07}), the reader is referred to the proof of the equation (3.4) in the proof of Lemma 3.1 of \cite{cornejo2016derived}.  \\
	
\noindent Proof of (\ref{070715_06}):
Let $a,b,c,d \in A$. Hence,
\noindent $(a \to b) \to ((c \to b) \to (d \to c)) $
$\overset{  (I) 
}{=}  (a \to b) \to (((d \to c)' \to c) \to (b \to (d \to c))')' $
$\overset{  (\ref{250615_01}) 
}{=}  (a \to b) \to ((d \to c) \to (b \to (d \to c))')' $
$\overset{  (\ref{291014_09}) 
}{=}  (a \to b) \to (((d \to c) \to b) \to (d \to c)) $
$\overset{  (\ref{291014_10}) 
}{=}  (a \to b) \to ((0 \to b) \to (d \to c)) $
$\overset{  (\ref{260615_07})
}{=}  (a \to b) \to (d \to c) $,

completing the proof.
\end{proof}

\begin{Lemma} \label{Lemma_0_C_equal_0}
	Let $\mathbf A \in \mathcal{I}_{2,0}$ such that $\mathbf A \models 0 \approx 0'$, then $\mathbf A \models 0 \to x \approx x$.
\end{Lemma}

\begin{proof}
	Let $a \in A$. Then  $a = 0' \to a = (0 \to 0) \to a = (0' \to 0) \to a \overset{\ref{general_properties_equiv} (\ref{LeftImplicationwithtilde})}{=} 0 \to a$.	
\end{proof}

\begin{Lemma} {\bf \cite{cornejo2016BolMoufang}}  \label{lemma_070616_01}
	Let $\mathbf A \in \mathcal{I}_{2,0}$ such that $\mathbf A \models 0 \to x \approx x$, then $\mathbf A \models (x \to y)' \approx x' \to y'$.
\end{Lemma}

Recall that the variety $\mathcal{S} = \mathcal{I}_{2,0} \cap \mathcal{MC}$,  
which was investigated in \cite{cornejo2015implication,cornejo2016BolMoufang}.  
Throughout this paper, $\mathbf{A} \in \mathcal{S}$.

\begin{Lemma} {\bf \cite{cornejo2016BolMoufang}} \label{properties_of_I20_MC}
	$\mathbf A$ satisfies:
	\begin{enumerate}
		\item $x \to (y \to z) \approx y \to (x \to z)$,  \label{310516_01}
		\item $x' \to y \approx y' \to x$.  \label{310516_02}
	\end{enumerate}
\end{Lemma}

\begin{Lemma}   \label{lemma_081116_A}
	Let $\mathbf A \in \mathcal{S}$ such that $\mathbf A \models 0 \to x \approx x \to x$, then 
	\begin{enumerate}
		\item $\mathbf A \models 0 \to (x \to x) \approx x \to x$. \label{231116_03}
		\item $\mathbf A \models 0 \to x' \approx 0 \to x$. \label{081116_06}
	\end{enumerate}
\end{Lemma}

\begin{proof}
Let $a \in A$. 
	\begin{enumerate}
		\item Observe that
%
%
\noindent $		0 \to (a \to a) $
$\overset{ hyp 
}{=}  (a \to a) \to (a \to a) $
$\overset{   \ref{general_properties2} (\ref{250315_05}) 
}{=}  a \to a $.
		\item 
\noindent $	0 \to a	$
$\overset{ hyp 
}{=}  a \to a $
$\overset{   \ref{properties_of_I20_MC} (\ref{310516_02}) 
}{=}  a' \to a' $
$\overset{ hyp 
}{=}  0 \to a' $.
	\end{enumerate}
This proves the lemma.
\end{proof}

\begin{Lemma}   \label{lemma_071116_A}
	Let $\mathbf A \in \mathcal{S}$ such that $\mathbf A \models 0 \to (x \to x) \approx x \to x$, then $\mathbf A$ satisfies the following identities:
	\begin{enumerate}
		\item $(x \to x) \to y' \approx ((x \to x) \to y)'$, \label{081116_02}
		\item $(x \to x) \to (y \to z) \approx ((x \to x) \to y) \to z$,\label{081116_03}
		\item $(x \to y) \to (x \to y) \approx (x \to x) \to (y \to y).$ \label{081116_04}
	\end{enumerate}
\end{Lemma}

\begin{proof} Items (\ref{081116_02}) and (\ref{081116_03}) follows from \cite[Lemma 3.4]{cornejo2016BolMoufang}. Let us prove (\ref{081116_04}). Let $a,b \in A$.
Observe that


\noindent $(a \to a) \to (0 \to a') $
$\overset{\ref{properties_of_I20_MC} (\ref{310516_01})}{=} 0 \to ((a \to a) \to a')$
$\overset{\ref{general_properties2} (\ref{281014_05})}{=} 0 \to ((a \to 0') \to a')$
$\overset{\ref{general_properties} (\ref{cuasiConmutativeOfImplic})}{=} 0 \to ((0 \to a') \to a')$
$\overset{\ref{properties_of_I20_MC} (\ref{310516_01})}{=} (0 \to a') \to (0 \to a')$
$\overset{\ref{general_properties2} (\ref{071114_04}) and (\ref{311014_06})}{=} 0 \to (a' \to a')$
$\overset{hyp}{=} a' \to a'$
$\overset{\ref{properties_of_I20_MC} (\ref{310516_02})}{=} a \to a$.
%
%
%
%
Hence, 
\begin{equation} \label{071116_05}
\mathbf A \models (x \to x) \to (0 \to x') \approx x \to x.
\end{equation}
Since

\noindent $(a \to b) \to (a \to b) $
$\overset{   \ref{properties_of_I20_MC} (\ref{310516_01}) 
}{=}  a \to ((a \to b) \to b) $
$\overset{   \ref{properties_of_I20_MC} (\ref{310516_02}) 
}{=}  a \to (b' \to (a \to b)') $
$\overset{   \ref{properties_of_I20_MC} (\ref{310516_01}) 
}{=}  b' \to (a \to (a \to b)') $
$\overset{   \ref{properties_of_I20_MC} (\ref{310516_02}) 
}{=}  b' \to ((a \to b) \to a') $
$\overset{  (I) 
}{=}  b' \to ((a \to a) \to (b \to a')')' $
$\overset{   \ref{properties_of_I20_MC} (\ref{310516_02}) 
}{=}  b' \to ((a \to a) \to (a \to b')')' $
$\overset{   (\ref{081116_02}) 
}{=}  b' \to ((a \to a) \to (a \to b'))'' $
$\overset{  x'' \approx x 
}{=}  b' \to ((a \to a) \to (a \to b')) $
$\overset{   \ref{properties_of_I20_MC} (\ref{310516_01}) 
}{=}  (a \to a) \to (b' \to (a \to b')) $
$\overset{   (\ref{081116_03}) 
}{=}  ((a \to a) \to b') \to (a \to b') $
$\overset{   \ref{properties_of_I20_MC} (\ref{310516_02}) 
}{=}  ((a \to a) \to b') \to (b \to a') $
$\overset{   \ref{properties_of_I20_MC} (\ref{310516_01}) 
}{=}  b \to (((a \to a) \to b') \to a') $
$\overset{   (\ref{081116_03}) 
}{=}  b \to ((a \to a) \to (b' \to a')) $
$\overset{   \ref{properties_of_I20_MC} (\ref{310516_02}) 
}{=}  b \to ((a \to a) \to (a \to b)) $
$\overset{   \ref{properties_of_I20_MC} (\ref{310516_01}) 
}{=}  (a \to a) \to (b \to (a \to b)) $
$\overset{   \ref{properties_of_I20_MC} (\ref{310516_01}) 
}{=}  (a \to a) \to (a \to (b \to b)) $
$\overset{   \ref{properties_of_I20_MC} (\ref{310516_02}) 
}{=}  (a \to a) \to ((b \to b)' \to a') $
$\overset{   (\ref{081116_03}) 
}{=}  (a \to a) \to ((b \to b) \to (0 \to a')) $
$\overset{   \ref{properties_of_I20_MC} (\ref{310516_01}) 
}{=}  (b \to b) \to ((a \to a) \to (0 \to a')) $
$\overset{  (\ref{071116_05}) 
}{=}  (b \to b) \to (a \to a) $,
we have
\begin{equation} \label{071116_06}
\mathbf A \models (x \to y) \to (x \to y) \approx (y \to y) \to (x \to x).
\end{equation}
Consequently,
\noindent $a \to ((a \to b) \to b) $
$\overset{   \ref{properties_of_I20_MC} (\ref{310516_01}) 
}{=}  (a \to b) \to (a \to b) $
$\overset{   \ref{properties_of_I20_MC} (\ref{310516_02}) 
}{=}  (b' \to a') \to (b' \to a') $
$\overset{  (\ref{071116_06}) 
}{=}  (a' \to a') \to (b' \to b') $
$\overset{   \ref{properties_of_I20_MC} (\ref{310516_02}). 
}{=}  (a \to a) \to (b \to b). $	
\end{proof}

\begin{Lemma} {\bf \cite{cornejo2016BolMoufang}} \label{310516_09}
	Let $\mathbf A \models x \to x \approx x$. Then $\mathbf A \models x' \approx x$. 
\end{Lemma}

\begin{Lemma}  \label{031116_01_Assoc1}
	Let $\mathbf A \models 0 \to x \approx x$. Then $\mathbf A \models x \to (x \to x) \approx (x \to x) \to x$. 
\end{Lemma}

\begin{proof}
	 By Lemma \ref{lemma_070616_01},
\begin{equation} \label{031116_02}
\mathbf A \models (x \to y)' \approx x' \to y'.
\end{equation}		
Let $a \in A$.	Then 
\noindent $(a \to a) \to a $
$\overset{   \ref{properties_of_I20_MC} (\ref{310516_02}) 
}{=}  a' \to (a \to a)' $
$\overset{  (\ref{031116_02}) 
}{=}  a' \to (a' \to a') $
$\overset{   \ref{general_properties2} (\ref{031114_04}) 
}{=}  a' \to a' $
$\overset{   \ref{properties_of_I20_MC} (\ref{310516_02}) 
}{=}  a \to a $
$\overset{   \ref{general_properties2} (\ref{031114_04})
}{=}  a \to (a \to a) $.
\end{proof}

Let $E$ be the set consisting of the terms:

$t_1(x,y,z,t) = ((x \to y) \to z) \to t$,

$t_2(x,y,z,t) = z \to ((y \to x) \to t)$,

$t_3(x,y,z,t) = (y \to x) \to (z \to t)$,

$t_4(x,y,z,t) = ((z \to y) \to x) \to t$ and 

$t_5(x,y,z,t) = (y \to z) \to (x \to t)$.

\begin{Lemma} \label{lemma_5terms}
	If $\mathbf A \models 0 \to x \approx x$, then $\mathbf A \models e_1 \approx e_2$ where $e_1, e_2 \in E$.
\end{Lemma}

\begin{proof}
	Since $\mathbf A \models 0 \to x \approx x$, by Lemma \ref{lemma_070616_01}, 
\begin{equation} \label{031116_03}
\mathbf A \models (x \to y)' \approx x' \to y'.
\end{equation}
Let $a,b,c,d \in A$. Then

\noindent $((a \to b) \to c) \to d $
$\overset{   \ref{properties_of_I20_MC} (\ref{310516_02}) 
}{=}  d' \to ((a \to b) \to c)' $
$\overset{  (\ref{031116_03}) 
}{=}  (d \to ((a \to b) \to c))' $
$\overset{   \ref{properties_of_I20_MC} (\ref{310516_02}) 
}{=}  (d \to (c' \to (a \to b)'))' $
$\overset{   \ref{properties_of_I20_MC} (\ref{310516_01}) 
}{=}  (c' \to (d \to (a \to b)'))' $
$\overset{  (\ref{031116_03}) 
}{=}  c'' \to (d \to (a \to b)')' $
$\overset{  x \approx x'' 
}{=}  c \to (d \to (a \to b)')' $
$\overset{   \ref{properties_of_I20_MC} (\ref{310516_02}) 
}{=}  c \to (d \to (b' \to a')')' $
$\overset{  (\ref{031116_03}) 
}{=}  c \to (d \to (b \to a))' $
$\overset{  (\ref{031116_03}) 
}{=}  c \to (d' \to (b \to a)') $
$\overset{   \ref{properties_of_I20_MC} (\ref{310516_02}) 
}{=}  c \to ((b \to a) \to d)  $
$\overset{   \ref{properties_of_I20_MC} (\ref{310516_01})  
}{=}  (b \to a) \to (c \to d) $,
proving   $t_1 \approx t_2$ and $t_1 \approx t_3$.

Next,
\noindent $((a \to b) \to c) \to d $
$\overset{   \ref{properties_of_I20_MC} (\ref{310516_02}) 
}{=}  (c' \to (a \to b)') \to d $
$\overset{  (\ref{031116_03}) 
}{=}  (c' \to (a' \to b')) \to d  $
$\overset{   \ref{properties_of_I20_MC} (\ref{310516_01}) 
}{=}  (a' \to (c' \to b')) \to d  $
$\overset{   \ref{properties_of_I20_MC} (\ref{310516_02}) 
}{=}  ((c' \to b')' \to a) \to d $
$\overset{  (\ref{031116_03})  
}{=}  ((c \to b) \to a) \to d $, proving  $t_1 \approx t_4$.

Also, we have that
\noindent $c \to ((b \to a) \to d) $
$\overset{   \ref{properties_of_I20_MC} (\ref{310516_02}) 
}{=}  c \to (d' \to (b \to a)') $
$\overset{   \ref{properties_of_I20_MC} (\ref{310516_02}) 
}{=}  c \to (d' \to (a' \to b')') $
$\overset{  (\ref{031116_03}) 
}{=}  c \to (d' \to (a \to b)) $
$\overset{   \ref{properties_of_I20_MC} (\ref{310516_01}) 
}{=}  c \to (a \to (d' \to b)) $
$\overset{   \ref{properties_of_I20_MC} (\ref{310516_01}) 
}{=}  a \to (c \to (d' \to b)) $
$\overset{   \ref{properties_of_I20_MC} (\ref{310516_01}) 
}{=}  a \to (d' \to (c \to b)) $
$\overset{   \ref{properties_of_I20_MC} (\ref{310516_02}) 
}{=}  a \to ((c \to b)' \to d) $
$\overset{   \ref{properties_of_I20_MC} (\ref{310516_02}) 
}{=}  a \to ((b' \to c')' \to d) $
$\overset{  (\ref{031116_03}) 
}{=}  a \to ((b \to c) \to d) $
$\overset{   \ref{properties_of_I20_MC} (\ref{310516_01}), 
}{=}  (b \to c) \to (a \to d) $, proving $t_2 \approx t_5$.
\end{proof}

\begin{Lemma} \label{lemma_091116_01}
	If $\mathbf A \models 0 \to x \approx x$, then $\mathbf A \models x \to ((x \to x) \to y) \approx (x \to (x \to x)) \to y$.
\end{Lemma}

\begin{proof}
	By Lemma \ref{lemma_5terms}, $\mathbf A \models e_1 \approx e_2$ for all $e_1, e_2 \in E$. Consider $a,b \in A$. We have that
\noindent $	a \to ((a \to a) \to b) $
$\overset{  
}{=}  t_1(a,a,a,b) $
$\overset{  
}{=}  t_4(a,a,a,b) $
$\overset{  
}{=}  ((a \to a) \to a) \to b $
$\overset{   \ref{031116_01_Assoc1} 
}{=}  (a \to (a \to a)) \to b $,
proving the lemma.
\end{proof}

\begin{Lemma} \label{lemma_identities_S_onevariable}
	$\mathbf A$ satisfies:
	\begin{enumerate}
		\item $(x \to x) \to (x \to x) \approx x \to (x \to (x \to x))$, \label{021216_01}
		\item $((x \to x) \to x) \to x \approx x \to (x \to (x \to x))$,  \label{021216_02}
		\item $(x \to y) \to (y \to z) \approx ((y \to x) \to y) \to z$, \label{311016_01}
		\item $y \to ((x \to y) \to z) \approx ((y \to x) \to y) \to z$, \label{311016_02}
		\item $(x \to x) \to (x \to y) \approx x \to ((x \to x) \to y)$, \label{311016_04}
		\item $x \to ((x \to x) \to y) \approx ((x \to x) \to x) \to y$, \label{311016_05}
		\item $x \to ((y \to x) \to x) \approx ((x \to y) \to x) \to x$, \label{311016_06}
		\item $x \to ((y \to x) \to y) \approx ((x \to y) \to x) \to y$. \label{311016_07}
	\end{enumerate} 
\end{Lemma}


\begin{proof} 
Let $a,b,c \in A$. 
	\begin{enumerate}
		\item  Observe that
	\noindent $		(a \to a) \to (a \to a) $
	$\overset{   \ref{properties_of_I20_MC} (\ref{310516_01}) 
	}{=}  a \to ((a \to a) \to a) $
	$\overset{   \ref{general_properties2} (\ref{291014_10}) 
	}{=}  a \to ((0 \to a) \to a) $
	$\overset{   \ref{properties_of_I20_MC} (\ref{310516_01}) 
	}{=}  (0 \to a) \to (a \to a) $
	$\overset{   \ref{general_properties2} (\ref{311014_05}) 
	}{=}  a \to (a \to a) $
	$\overset{   \ref{general_properties2} (\ref{031114_04}) 
	}{=}  a \to (a \to (a \to a)) $.
		\item 
\noindent $		((a \to a) \to a) \to a $
$\overset{   \ref{general_properties2} (\ref{291014_10}) 
}{=}  ((0 \to a) \to a) \to a $
$\overset{   \ref{general_properties2} (\ref{291014_08}) 
}{=}  a \to a $
$\overset{   \ref{general_properties2} (\ref{031114_04}) 
}{=}  a \to (a \to a) $
$\overset{   \ref{general_properties2} (\ref{031114_04})
}{=}  a \to (a \to (a \to a)) $.	
		
		\item 
\noindent $	(a \to b) \to (b \to c)	$
$\overset{   \ref{general_properties2} (\ref{250615_04}) 
}{=}  (0 \to a') \to (b \to c) $
$\overset{  (I) 
}{=}  (((b \to c)' \to 0) \to (a' \to (b \to c))')' $
$\overset{  x \approx x'' 
}{=}  ((b \to c) \to (a' \to (b \to c))')' $
$\overset{   \ref{properties_of_I20_MC} (\ref{310516_02}) 
}{=}  ((a' \to (b \to c)) \to (b \to c)')' $
$\overset{   \ref{properties_of_I20_MC} (\ref{310516_01}) 
}{=}  ((b \to (a' \to c)) \to (b \to c)')' $
$\overset{   \ref{properties_of_I20_MC} (\ref{310516_02}) 
}{=}  ((b \to (c' \to a)) \to (b \to c)')' $
$\overset{   \ref{properties_of_I20_MC} (\ref{310516_01}) 
}{=}  ((c' \to (b \to a)) \to (b \to c)')' $
$\overset{  (I)
}{=}  ((b \to a) \to b) \to c $.
		
	\item 
	\noindent $	b \to ((a \to b) \to c)	$
	$\overset{   \ref{properties_of_I20_MC} (\ref{310516_01}) 
	}{=}  (a \to b) \to (b \to c) $
	$\overset{  (\ref{311016_01}) 
	}{=}  ((b \to a) \to b) \to c $.
\end{enumerate} 
		
For (\ref{311016_04}) use Lemma \ref{properties_of_I20_MC} (\ref{310516_01}), and
for (\ref{311016_05}) use Lemma \ref{properties_of_I20_MC} (\ref{310516_01}) and item  (\ref{311016_01}).
 (\ref{311016_06}) is a special case of (\ref{311016_02}).  
Finally, one can use Lemma \ref{properties_of_I20_MC} (\ref{310516_01}) and item  (\ref{311016_01}) to prove (\ref{311016_07}).
\end{proof}





\section{Weak associative laws of length 3}
In this section we examine all the weak associative laws of length 3.


\subsection{With one variable:}  

\noindent The only word of length 3 with 1 variable is: 

A: $\langle x, x, x \rangle$. 

\noindent Ways in which the word A can be bracketed (where $a$ is just a place holder) are:

\begin{minipage}{0.5 \textwidth}
\begin{multicols}{2}
1: $a \to (a \to a)$, 

2: $(a \to a) \to a$. 
\end{multicols}
\end{minipage}

\medskip \noindent The only weak associative identitiy in this category is: \medskip  

1: (31A12)  $x \to (x \to x)$ $\approx$ $(x \to x) \to x$.

\subsection{With 2 variables:}  


\noindent Possible words of length 3 with 2 variables are: 

\begin{minipage}{0.8 \textwidth}
\begin{multicols}{3}
A: $\langle x, x, y \rangle$, 

B: $\langle x, y, x\rangle$, 

C: $\langle x, y, y\rangle$. 
\end{multicols}
\end{minipage}

\medskip \noindent Ways in which a word of size 3 can be bracketed:\medskip  

\begin{minipage}{0.8 \textwidth}
\begin{multicols}{2}
1: $a \to (a \to a)$, 

2: $(a \to a) \to a$. 	
\end{multicols}
\end{minipage}

\medskip \noindent 
The weak associative identities in this category are:

\medskip
1: (LALT)  $x \to (x \to y)$ $\approx$ $(x \to x) \to y$ ({\it the left-alternative law})

2: (FLEX)  $x \to (y \to x)$ $\approx$ $(x \to y) \to x$  (the {\it flexible law})

3: (RALT)  $x \to (y \to y)$ $\approx$ $(x \to y) \to y$ ({\it the right-alternative law})\\

We should note that we did not follow our convention in this case, since these identities are well known by the above names. 
We let $\mathcal{LALT}$, $\mathcal{FLEX}$ and $\mathcal{RALT}$ denote, respectively, the subvarieties $\mathcal S$ defined by (LALT), \linebreak
(FLEX)  and (RALT).

\begin{Lemma} \label{Lemma_SL_subsetFLEX_RALT_LALT}
	$\mathcal{SL} \subseteq \mathcal{FLEX} \cap \mathcal{RALT} \cap \mathcal{LALT}$. 
\end{Lemma}
\begin{proof}
It is easy to verify that $\mathbf{2_s}$ satisfies  (FLEX), (RALT) and (LALT).
	The lemma follows since $\mathbf{V(2_s)} = \mathcal{SL}$.  
\end{proof}

\begin{Theorem} \label{theorem_131216_01}
	$\mathcal{FLEX} = \mathcal{RALT} = \mathcal{LALT} =  \mathcal{SL}$.
	
\end{Theorem}

\begin{proof} Let $\mathbf A \in \mathcal{LALT} \cup \mathcal{FLEX} \cup  \mathcal{RALT}$  and $a \in A$.
	
	Let $\mathbf A \in \mathcal{LALT}$.  Then 
\noindent $a	$
$\overset{  x \approx x'' 
}{=}  a'' $
$\overset{   \ref{general_properties_equiv} (\ref{LeftImplicationwithtilde}) 
}{=}  (a'' \to a')' $
$\overset{  x \approx x'' 
}{=}  (a \to a')' $
$\overset{  
}{=}  (a \to (a \to 0))' $
$\overset{  (LALT) 
}{=}  ((a \to a) \to 0)' $
$\overset{  
}{=}  (a \to a)'' $
$\overset{  x \approx x''
}{=}  a \to a $.

Next, let $\mathbf A \in \mathcal{FLEX}$.  Then 
\noindent $a $
$\overset{   \ref{general_properties_equiv} (\ref{LeftImplicationwithtilde}) 
}{=}  a' \to a $
$\overset{   \ref{general_properties_equiv} (\ref{LeftImplicationwithtilde}) 
}{=}  (a'' \to a') \to a $
$\overset{  x \approx x'' 
}{=}  (a \to a') \to a  $
$\overset{  (FLEX) 
}{=}  a \to (a' \to a) $
$\overset{   \ref{general_properties_equiv} (\ref{LeftImplicationwithtilde})
}{=}  a \to a $.

Finally, let $\mathbf A \in \mathcal{RALT}$.  Then 
\noindent $a \to a $
$\overset{  x \approx x'' 
}{=}  a'' \to a $
$\overset{  
}{=}  ((a \to 0) \to 0) \to a $
$\overset{  (RALT) 
}{=}  (a \to (0 \to 0)) \to a  $
$\overset{  
}{=}  (a \to 0') \to a $
$\overset{   \ref{general_properties2} (\ref{281014_05}) 
}{=}  (a \to a') \to a $
$\overset{  x \approx x'' 
}{=}  (a'' \to a') \to a $
$\overset{   \ref{general_properties_equiv} (\ref{LeftImplicationwithtilde}) 
}{=}  a' \to a $
$\overset{   \ref{general_properties_equiv} (\ref{LeftImplicationwithtilde})
}{=}  a $.
Therefore, $\mathbf A \models x \to x \approx x$. By Lemma \ref{310516_09}, $\mathbf A \models x \approx x'$ and consequently, $\mathbf A \in \mathcal{SL}$.  Hence the result is valid in view of Lemma \ref{Lemma_SL_subsetFLEX_RALT_LALT}.
\end{proof}

\subsection{With 3 variables:}

\medskip


\noindent The only word of length 3 with 3 variables is:
\medskip  

A: $\langle x, y, z\rangle$. 

\medskip \noindent Ways in which a word of length 3 can be bracketed:\medskip  

\begin{minipage}{0.5 \textwidth}
\begin{multicols}{2}
	1: $a \to (a \to a)$, 
	
	2: $(a \to a) \to a$. 	
\end{multicols}	
\end{minipage}
%

\medskip \noindent 
The only weak associative identitiy in this category is: \medskip  

($33\mathcal{A}12$)  $x \to (y \to z)$ $\approx$ $(x \to y) \to z$  (associative law).






\begin{Lemma} \label{Old_Lemma_10.3(3)}
$33\mathcal{A}12 = \mathcal{SL}$.
\end{Lemma}

\begin{proof}
 By \cite{cornejo2015implication}, $\mathcal{SL} \subseteq 33\mathcal{A}12$. Hence let us consider $\mathbf A \in 33\mathcal{A}12$ and $a \in A$. Observe that $0' = 0 \to 0' = 0 \to (0 \to 0) = (0 \to 0) \to 0 = 0' \to 0 = 0$. Then 
\begin{equation} \label{101116_01}
\mathbf A \models 0' \approx 0.
\end{equation}
Therefore,
\noindent $a $
$\overset{   \ref{general_properties_equiv} (\ref{LeftImplicationwithtilde}) 
}{=}  a' \to a $
$\overset{  
}{=}  (a \to 0) \to a $
$\overset{ (33A12) 
}{=}  a \to (0 \to a) $
$\overset{  (\ref{101116_01}) 
}{=}  a \to (0' \to a) $
$\overset{   \ref{general_properties_equiv} (\ref{TXX}) 
}{=}  a \to a $.
Consequently, $\mathbf A \in \mathcal{SL}$ by Lemma \ref{310516_09}.
\end{proof}

\section{Weak associative laws with length 4 and with 1 variable.}

The only word of length 4 with 1 variable is:

\medskip  
A: $\langle x, x, x, x \rangle$.

\medskip \noindent Ways in which a word of length 4 can be bracketed are:\medskip  

\begin{minipage}{1 \textwidth}
\begin{multicols}{3}
1: $a \to (a \to (a \to a))$, 

2: $a \to ((a \to a) \to a)$, 

3: $(a \to a) \to (a \to a)$, 

4: $(a \to (a \to a)) \to a$, 

5: $((a \to a) \to a) \to a$.
\end{multicols}
\end{minipage}

\medskip

By now, we believe that reader is well acquainted with our notation for identities.  So, 
 we will, no longer, present the list of the identities in this and the remaining categories. 

\begin{Lemma} \label{lemma_131216_01} The following hold:
	\begin{enumerate}
		\item $41\mathcal A12 = 41\mathcal A13$, \label{281016_01}
		\item $41\mathcal A24 = 41\mathcal A34$, \label{281016_02}
		\item $41\mathcal A25 = 41\mathcal A35$, \label{281016_03}
		\item $41\mathcal A23 = \mathcal S$, \label{281016_04}
		\item $41\mathcal A13 = 41\mathcal A15 = 41\mathcal A35 = \mathcal S$,  \label{281016_05}
		\item $41\mathcal A14 = 41\mathcal A34 = 41\mathcal A45$. \label{281016_06}
	\end{enumerate}
\end{Lemma}

\begin{proof}
	
 Items (\ref{281016_01}), (\ref{281016_02}),  (\ref{281016_03}) and (\ref{281016_04})
 follow from 
 Lemma \ref{properties_of_I20_MC} (\ref{310516_01}).

For (\ref{281016_05}), observe that $\mathcal S \subseteq 41\mathcal A13$ by Lemma \ref{lemma_identities_S_onevariable} (\ref{021216_01}),  
$\mathcal S \subseteq 41\mathcal A15$ by  Lemma \ref{lemma_identities_S_onevariable} (\ref{021216_02}),  
and $\mathcal S \subseteq 41\mathcal A35$ by Lemma \ref{lemma_identities_S_onevariable} (\ref{311016_05}) and (\ref{311016_04}).

To prove (\ref{281016_06}), $41 \mathcal A14 \subseteq 41\mathcal A34$ follows from (41A14) and 	Lemma \ref{lemma_identities_S_onevariable} (\ref{021216_01}). From (41A34), Lemma \ref{properties_of_I20_MC} (\ref{310516_01}) and  Lemma \ref{lemma_identities_S_onevariable} (\ref{311016_05}) we have that $41 \mathcal A34 \subseteq 41\mathcal A45$. The inclusion $41 \mathcal A45 \subseteq 41\mathcal A14$ follows from (41A45) and Lemma \ref{lemma_identities_S_onevariable} (\ref{021216_02}).

\end{proof}


\begin{Lemma} \label{Old_Lemma 10.3(1)}
$41\mathcal A14 = 31\mathcal A12$.
\end{Lemma}

\begin{proof}
 It is enough to observe that, 
 in view of Lemma \ref{general_properties2} (\ref{031114_04}), the identities defining these varieties are both equivalent to $x \to x \approx (x \to x) \to x$.
\end{proof}

\section{Weak associative laws with length 4 and with 2 variables.}

\medskip 
 Possible words of length 4 with 2 variables are: 

\medskip 

\begin{minipage}{0.8 \textwidth} 
\begin{multicols}{4}
A: $\langle x, x, x, y\rangle$, 

B: $\langle x, x, y, x\rangle$, 

C: $\langle x, x, y, y\rangle$, 

D: $\langle x, y, x, x\rangle$, 

E: $\langle x, y, x, y \rangle$, 

F: $\langle x, y, y, x\rangle$, 

G: $\langle x, y, y, y\rangle$. 
\end{multicols}
\end{minipage}

\medskip \noindent Ways in which a word of size 4 can be bracketed are: \medskip

\begin{minipage}{0.8 \textwidth}
	\begin{multicols}{2}
1: $a \to (a \to (a \to a))$, 

2: $a \to ((a \to a) \to a)$, 

3: $(a \to a) \to (a \to a)$, 

4: $(a \to (a \to a)) \to a$, 

5: $((a \to a) \to a) \to a$.
\end{multicols}
\end{minipage}
 \\

It is easy to see that there are 70 identities in this category, and accordingly there are 70 subvarieties of $\mathcal S$ defined by them.  As not all of them will be distinct, we will partition these 70 varieties subject to the relation of ``being equal''.  We will choose one representative from each block of the partition.
As will be shown below, some of them will equal $\mathcal{SL}$ and some others will equal $\mathcal{S}$.  
We will use  
$\mathcal{SL}$ and $\mathcal{S}$ as representatives of the blocks whose varieties are equal to $\mathcal{SL}$ and $\mathcal{S}$ respectively. 
\begin{Lemma} \label{including_relations_of_SL_42}
	$\mathcal{SL} \subseteq 42\mathcal Xij$, for $\mathcal X \in \{A,B,C,D,E,F,G\}$ and for all $i,j$ such that $1 \leq i < j \leq 5$.
\end{Lemma}
\begin{proof}
	By a routine computation, it is easy to check that $\mathbf{2_s} \in  42\mathcal Xij$ for all $1 \leq i < j \leq 5$.
	Then the proof is complete since $\mathcal{V}(\mathbf{2_s}) = \mathcal{SL}$, in view of 
	\cite[Corollary 10.4]{cornejo2015implication}.	
\end{proof}

\begin{Lemma} \label{lemma_081116_02}  We have
	\begin{enumerate}
		\item $42\mathcal A23 = 42\mathcal A25 = 42\mathcal A35 =42\mathcal D25 = 42\mathcal E25 = \mathcal S$, \label{311016_03}
		\item $42\mathcal A12 = 42\mathcal A13 = 42\mathcal A15$ and $42\mathcal A24 = 42\mathcal A34$, \label{231116_06}
		\item $42\mathcal D12 = 42\mathcal D15$ and $42\mathcal E12 = 42\mathcal E15$. \label{231116_07}
	\end{enumerate}
\end{Lemma}

\begin{proof}
		\begin{enumerate}
			\item From Lemma \ref{lemma_identities_S_onevariable} (\ref{311016_04}) and (\ref{311016_05}) we have that $$42\mathcal A23 = 42\mathcal A25 = 42\mathcal A35 = \mathcal S.$$ Using Lemma \ref{lemma_identities_S_onevariable} (\ref{311016_06}) and (\ref{311016_07}) we can conclude that $$42\mathcal D25 = 42\mathcal E25 = \mathcal S.$$
			\item Follows from Lemma \ref{lemma_identities_S_onevariable} (\ref{311016_04}) and (\ref{311016_05}).
			\item Follows from Lemma \ref{lemma_identities_S_onevariable} (\ref{311016_06}) and (\ref{311016_07}).
		\end{enumerate}
The proof is complete.
\end{proof}

%
%

Let $M$ be the set consisting of the following pairs of varieties:  $(42\mathcal B12,42\mathcal D13)$, $(42\mathcal B13,42\mathcal G12)$, $(42\mathcal C12,42\mathcal E13)$, $(42\mathcal C13,42\mathcal F12)$, $(42\mathcal D12,42\mathcal G13)$, $(42\mathcal D14,42\mathcal G14)$, $(42\mathcal D24,42\mathcal G34)$, $(42\mathcal E12,42\mathcal F13)$, $(42\mathcal E14,42\mathcal F14)$, $(42\mathcal E24,42\mathcal F34)$.

\begin{Lemma} \label{lemma_081116_03}
	If $(\mathcal A, \mathcal B) \in M$ then $\mathcal A = \mathcal B$.
\end{Lemma}

\begin{proof}
	Follows from Lemma \ref{properties_of_I20_MC} (\ref{310516_01}). Rename the variables, if necessary. 
	
\end{proof}

Let $N$ be the set consisting of the following varieties $42\mathcal A14$, $42\mathcal B14$, $42\mathcal B24$, $42\mathcal B34$, $42\mathcal B45$, $42\mathcal C14$, $42\mathcal C24$, $42\mathcal C34$, $42\mathcal C45$, $42\mathcal D14$, $42\mathcal D34$, $42\mathcal E12$, $42\mathcal E14$, $42\mathcal E23$, $42\mathcal E24$, $42\mathcal E34$, $42\mathcal E35$, $42\mathcal E45$, $42\mathcal F15$, $42\mathcal F23$, $42\mathcal F24$, $42\mathcal F45$, $42\mathcal G23$, $42\mathcal G24$, $42\mathcal G35$ , $42\mathcal G45$.

\begin{Theorem} \label{theorem_42X_equalTo_SL}
	If $\mathcal X \in N$ then $\mathcal X = \mathcal{SL}$.
\end{Theorem}

\begin{proof}
		In the proof below the following list of statements will be useful.
		
		\medskip
		
		\noindent \begin{minipage}{0.1 \textwidth}
			($\ast$)
		\end{minipage}
		\begin{minipage}{0.9 \textwidth}
			The identity (\ref{eq_I20}), 
			Lemma \ref{general_properties_equiv} (\ref{TXX}),
			Lemma \ref{general_properties_equiv} (\ref{LeftImplicationwithtilde}), 
			Lemma \ref{general_properties2} (\ref{281014_05}),
			Lemma \ref{general_properties3} (\ref{031114_04}) and 
			Lemma \ref{properties_of_I20_MC} (\ref{310516_01}).
		\end{minipage}

\medskip
	Let $\mathcal X \in N$.  In view of Lemma \ref{including_relations_of_SL_42}, it suffices to prove that $\mathcal X \subseteq \mathcal{SL}$.  In fact, by Lemma \ref{lemma_SL_I10_C}, it suffices to prove that $\mathcal X \models x' \approx x$. 

	 Let $\mathbf A \in \mathcal X$ and let 
	$a \in A$.
	
	To facilitate a uniform presentation (and to make the proof shorter), we introduce the following notation, where $x_0, y_0 \in \mathbf A$:\\
	
	The notation	
	$$ \mathcal X/ x_0, y_0$$ denotes the following 
	statement: 
	
	\begin{center}	 
		``In the identity (X) that defines the variety $\mathcal X$, relative to $\mathcal S$, if we assign $x:= x_0, y:= y_0$ 
		(and simplify it using the list ($\ast$)), 
		then $\mathbf{A} \models  x \to x \approx x$''.\\ 
	\end{center}
	
	Firstly, we consider the varieties associated with the following statements:
	
	\begin{multicols}{2}
		
		\begin{enumerate}
			
\item $42\mathcal A14/ a, 0$,  
\item $42\mathcal B14/ a, a'$,
\item $42\mathcal B24/ a, a'$,
\item $42\mathcal C24/ a, 0$,
\item $42\mathcal C45/ a, 0$,
\item $42\mathcal E12/ a', a$,
\item $42\mathcal E23/ a', 0$, \label{061216_01}
\item $42\mathcal E35/ a', 0$,
\item $42\mathcal F24/ a, 0$,
\item $42\mathcal F45/ a, 0$.
			
		\end{enumerate}
		
	\end{multicols}
We prove (\ref{061216_01}) as an illustration.

\noindent $a $
$\overset{   \ref{general_properties_equiv} (\ref{LeftImplicationwithtilde}) 
}{=}  a' \to a $
$\overset{   \ref{general_properties_equiv} (\ref{LeftImplicationwithtilde}) 
}{=}  (a'' \to a') \to a $
$\overset{  x \approx x'' 
}{=}  (a \to a') \to a $
$\overset{   \ref{general_properties2} (\ref{291014_10}) 
}{=}  (0 \to a') \to a $
$\overset{   \ref{properties_of_I20_MC} (\ref{310516_02}) 
}{=}  a' \to (0 \to a')' $
$\overset{  
}{=}  a' \to ((0 \to a') \to 0) $
$\overset{  (42E23) 
}{=}  (a' \to 0) \to (a' \to 0) $
$\overset{  
}{=}  a'' \to a'' $
$\overset{  x \approx x''
}{=}  a \to a $.
	
Hence the statement (7) is true.  Similarly, one can verify the rest of the above statements is true, from which it follows that, in each of the above cases,
	$\mathcal X \models x \to x \approx x$.  
	
	Then, applying Lemma \ref{310516_09}, we get that $\mathcal X \models x' \approx x$.  \\	 

	The notation, where $x_0, y_0,  x_1, y_1, \in \mathbf{A}$, 
	
	$$\mathcal X / x_0, y_0 / x_1, y_1\ /p \approx q$$
	is an abbreviation for the following statement: \\
	
	``In the identity  (X) that defines the variety $\mathcal X$, relative to $\mathcal S$, if we assign $x:= x_0, y:= y_0$,
	(and simplify (X) using the appropriate lemmas from the list ($\ast$)), we obtain that $\mathbf A \models 0' \approx 0$; 
	and then we assign $x:= x_1, y:= y_1$ in the identity (X) (and simplify it using $0'=0$ and the list ($\ast$)), then $\mathbf A \models p \approx q$.''\\

	Secondly, consider the varieties associated with the following statements:
	\begin{multicols}{2}
		
		\begin{enumerate}
\item $42\mathcal B45 / 0, 0 / a, a'\ / x \to x \approx x$,
\item $42\mathcal C34 / 0, 0 / a, 0\ / x \to x \approx x$,	\label{061216_02}
\item $42\mathcal E14 / 0, 0 / a, 0\ / x \to x \approx x$,
\item $42\mathcal E24 / 0, 0 / a, 0 \ / x \to x \approx x$,
\item $42\mathcal E34 / 0, 0 / 0, a\ / x \to x \approx x$,
\item $42\mathcal E45 / 0, 0 / a, 0\ / x \to x \approx x$,
\item $42\mathcal F23 / 0, 0' / a, 0\ / x \to x \approx x$. 
		
		\end{enumerate}
	\end{multicols}
As a sample,  we prove (\ref{061216_02}) below:

\noindent $0 $
$\overset{   \ref{general_properties_equiv} (\ref{LeftImplicationwithtilde}) 
}{=}  0' \to 0 $
$\overset{   \ref{general_properties_equiv} (\ref{LeftImplicationwithtilde}) 
}{=}  (0'' \to 0') \to 0 $
$\overset{  
}{=}  (0 \to 0') \to 0  $
$\overset{  
}{=}  (0 \to (0 \to 0)) \to 0 $
$\overset{  (42C34) 
}{=}  (0 \to 0) \to (0 \to 0) $
$\overset{  
}{=}  0' \to 0' $
$\overset{   \ref{general_properties_equiv} (\ref{TXX}) 
}{=}  0'  $
then
\noindent $a \to a $
$\overset{   \ref{general_properties_equiv} (\ref{TXX}) 
}{=}  0' \to (a \to a) $
$\overset{  0 \approx 0' 
}{=}  0 \to (a \to a) $
$\overset{   \ref{properties_of_I20_MC} (\ref{310516_02}) 
}{=}  0 \to (a' \to a') $
$\overset{   \ref{general_properties2} (\ref{191114_05}) 
}{=}  0 \to (a \to a)' $
$\overset{  
}{=}  0 \to ((a \to a) \to 0) $
$\overset{   \ref{properties_of_I20_MC} (\ref{310516_01}) 
}{=}  (a \to a) \to (0 \to 0) $
$\overset{  (42C34) 
}{=}  (a \to (a \to 0)) \to 0 $
$\overset{  
}{=}  (a \to a')' $
$\overset{   \ref{general_properties_equiv} (\ref{LeftImplicationwithtilde}) 
}{=}  a'' $
$\overset{  x \approx x'' 
}{=}  a $.
Hence, (\ref{061216_02}) holds.  Similarly, one can verify that each of the above statements is true.  Hence, it follows that  in each case $\mathcal{X} \models  x \to x \approx x$.  
	Then, applying Lemma \ref{310516_09}, we get that $\mathcal X \models x' \approx x$.  	 
	
	Thirdly, consider the varieties associated with the following statements:
	
	\begin{multicols}{2}
		\begin{enumerate}
\item $42\mathcal C14 / 0, 0 / a, 0\ / x' \approx x$,
\item $42\mathcal D14 / 0, 0 / 0, a \ / x' \approx x$,	
\item $42\mathcal D34 / 0, 0 / 0, a\ / x' \approx x$,	
\item $42\mathcal G23 / 0', 0 / a, 0\ / x' \approx x$,	
\item $42\mathcal G24 / 0, 0 / a, 0\ / x' \approx x$,	
\item $42\mathcal G35 / 0', 0 / a, 0 \ / x' \approx x$,
\item $42\mathcal G45 / 0, 0 / a, 0\ / x'\approx x$.
		\end{enumerate}
	\end{multicols}
	
	It is easy to verify that the above statements are true.
	Hence, it follows in each of the above cases that $\mathcal X \models x' \approx x.$
	
	Thus, the varieties still left to consider are $42\mathcal B34$ and $42\mathcal F15$.

Let $\mathbf A \in 42\mathcal B34$ and let $a \in A$. Since
\noindent $a $
$\overset{  x \approx x'' 
}{=}  a'' $
$\overset{   \ref{general_properties_equiv} (\ref{TXX}) 
}{=}  (0' \to a')' $
$\overset{  
}{=}  ((0 \to 0) \to (a \to 0))' $
$\overset{  (42B34) 
}{=}  ((0 \to (0 \to a)) \to 0)' $
$\overset{   \ref{general_properties2} (\ref{031114_04}) 
}{=}  ((0 \to a) \to 0)' $
$\overset{  x \approx x'' 
}{=}  0 \to a $,
we have that
\begin{equation} \label{021116_01}
\mathbf A \models x \approx 0 \to x.
\end{equation}
Therefore,
\noindent $a \to a $
$\overset{  (\ref{021116_01}) 
}{=}  (0 \to a) \to a $
$\overset{   \ref{general_properties2} (\ref{291014_10}) 
}{=}  (a \to a) \to a $
$\overset{   \ref{general_properties_equiv} (\ref{TXX}) 
}{=}  (a \to a) \to (0' \to a) $
$\overset{  (42B34) 
}{=}  (a \to (a \to 0')) \to a $
$\overset{   \ref{general_properties2} (\ref{031114_04}) 
}{=}  (a \to 0') \to a $
$\overset{   \ref{general_properties2} (\ref{281014_05}) 
}{=}  (a \to a') \to a $
$\overset{  x \approx x'' 
}{=}  (a'' \to a') \to a $
$\overset{   \ref{general_properties_equiv} (\ref{LeftImplicationwithtilde}) 
}{=}  a' \to a $
$\overset{   \ref{general_properties_equiv} (\ref{LeftImplicationwithtilde})} 
{=}  a $.
Hence 
$$
\mathbf A \models x \approx x \to x.
$$
Then, applying Lemma \ref{310516_09}, we get that $\mathbf A \models x' \approx x$.

Let $\mathbf A \in 42\mathcal F15$ and let $a \in A$. If we replace $x:= 0$ and $y:= 0'$ we obtain that 

\begin{equation} \label{021116_02}
\mathbf A \models 0 \approx 0'.
\end{equation}
 Since
\noindent $a $
$\overset{   \ref{general_properties_equiv} (\ref{TXX}) 
}{=}  0' \to a $
$\overset{  (\ref{021116_02}) 
}{=}  0 \to a $
$\overset{   \ref{general_properties_equiv} (\ref{LeftImplicationwithtilde}) 
}{=}  0 \to (a' \to a) $
$\overset{   x \approx x'' 
}{=}  0 \to (a' \to a'') $
$\overset{   
}{=}  0 \to (a' \to (a' \to 0)) $
$\overset{  (42F15) 
}{=}  ((0 \to a') \to a') \to 0 $
$\overset{  (\ref{021116_02}) 
}{=}  ((0' \to a') \to a') \to 0 $
$\overset{   \ref{general_properties_equiv} (\ref{TXX}) 
}{=}  (a' \to a') \to 0 $
$\overset{   \ref{properties_of_I20_MC} (\ref{310516_02}) 
}{=}  (a \to a) \to 0 $,
the identity
\begin{equation} \label{021116_03}
\mathbf A \models x \approx (x \to x)'
\end{equation}
holds in $\mathbf A$. Hence
\noindent $a' $
$\overset{  (\ref{021116_03}) 
}{=}  (a' \to a')' $
$\overset{   \ref{properties_of_I20_MC} (\ref{310516_02}) 
}{=}  (a \to a)' $
$\overset{  (\ref{021116_03}) 
}{=}  a $,
proving the lemma.
\end{proof}

Let $I$ be the set consisting of the following identities: $(42A24)$, $(42A45)$, $(42B13)$, $(42B15)$, $(42B23)$, $(42B25)$, $(42D24)$, $(42D45)$, $(42F25)$ and $(42G15)$.

\begin{Lemma} \label{lema_1_identity}
	Let $(x) \in I$. If $\mathbf A \models (x)$ then $A \models 0 \to x \approx x$.
\end{Lemma}

\begin{proof}
Let $a \in A$.  
 We consider the following cases:\\
Case 1:  $\mathbf A \models (42A24)$.  Then 
\noindent $a $
$\overset{   \ref{general_properties_equiv} (\ref{TXX}) 
}{=}  0' \to a $
$\overset{   \ref{general_properties_equiv} (\ref{LeftImplicationwithtilde}) 
}{=}  (0'' \to 0') \to a $
$\overset{  
}{=}  (0 \to 0') \to a $
$\overset{  
}{=}  (0 \to (0 \to 0)) \to a $
$\overset{  (42A24) 
}{=}  0 \to ((0 \to 0) \to a) $
$\overset{  
}{=}  0 \to (0' \to a) $
$\overset{   \ref{general_properties_equiv} (\ref{TXX}) \\
}{=}  0 \to a $.\\
Case 2:  $\mathbf A \models (42A45)$.  We have
\noindent $a $
$\overset{   \ref{general_properties_equiv} (\ref{TXX}) 
}{=}  0' \to a $
$\overset{   \ref{general_properties_equiv} (\ref{LeftImplicationwithtilde}) 
}{=}  (0'' \to 0') \to a $
$\overset{  
}{=}  (0 \to 0') \to a $
$\overset{  
}{=}  (0 \to (0 \to 0)) \to a $
$\overset{  (42A45) 
}{=}  ((0 \to 0) \to 0) \to a $
$\overset{  
}{=}  (0' \to 0) \to a $
$\overset{   \ref{general_properties_equiv} (\ref{TXX}) 
}{=}  0 \to a $.\\
Case 3:  $\mathbf A \models (42B13)$.  In this case we get 
\noindent $a $
$\overset{  x \approx x'' 
}{=}  a'' $
$\overset{  
}{=}  a' \to 0 $
$\overset{   \ref{general_properties_equiv} (\ref{TXX}) 
}{=}  0' \to (a' \to 0) $
$\overset{  
}{=}  (0 \to 0) \to (a' \to 0) $
$\overset{  (42B13) 
}{=}  0 \to (0 \to (a' \to 0)) $
$\overset{   \ref{general_properties2} (\ref{031114_04}) 
}{=}  0 \to (a' \to 0) $
$\overset{  
}{=}  0 \to a'' $
$\overset{  x \approx x''
}{=}  0 \to a $.\\
Case 4:  $\mathbf A \models (42B15)$.  Then
\noindent $a $
$\overset{  x \approx x'' 
}{=}  a'' $
$\overset{  
}{=}  a' \to 0 $
$\overset{   \ref{general_properties_equiv} (\ref{TXX}) 
}{=}  (0' \to a') \to 0 $
$\overset{  
}{=}  ((0 \to 0) \to a') \to 0 $
$\overset{  (42B15) 
}{=}  0 \to (0 \to (a' \to 0)) $
$\overset{   \ref{general_properties2} (\ref{031114_04}) 
}{=}  0 \to (a' \to 0) $
$\overset{  
}{=}  0 \to a'' $
$\overset{  x \approx x'' 
}{=}  0 \to a $.\\
If $\mathbf A \models (42B23)$,
\noindent $a $
$\overset{  x \approx x'' 
}{=}  a'' $
$\overset{  
}{=}  a' \to 0 $
$\overset{   \ref{general_properties_equiv} (\ref{TXX}) 
}{=}  0' \to (a' \to 0) $
$\overset{  
}{=}  (0 \to 0) \to (a' \to 0) $
$\overset{  (42B23) 
}{=}  0 \to ((0 \to a') \to 0) $
$\overset{   \ref{general_properties2} (\ref{031114_07}) 
}{=}  0 \to a'' $
$\overset{  x \approx x''
}{=}  0 \to a $.\\
Case 5:  $\mathbf A \models (42B25)$.  One has 
\noindent $a $
$\overset{  x \approx x'' 
}{=}  a'' $
$\overset{  
}{=}  a' \to 0 $
$\overset{   \ref{general_properties_equiv} (\ref{TXX}) 
}{=}  (0' \to a') \to 0 $
$\overset{  
}{=}  ((0 \to 0) \to a') \to 0 $
$\overset{  (42B25) 
}{=}  0 \to ((0 \to a') \to 0) $
$\overset{   \ref{general_properties2} (\ref{031114_07}) 
}{=}  0 \to a'' $
$\overset{  x \approx x''
}{=}  0 \to a $.\\
Case 6:  $\mathbf A \models (42D24).$  Then, setting $x:=0$ and $y := 0$ in the identity $(42D24)$, we obtain that $0' = 0$. Then apply Lemma \ref{Lemma_0_C_equal_0}.\\
Case 7:  $\mathbf A \models (42D45)$,  Then, set  $x:=0$ and $y := 0$ in the identity $(42D45)$ to obtain that $0' = 0$.  Now, apply Lemma \ref{Lemma_0_C_equal_0}.\\
Case 8:  $\mathbf A \models (42F25)$. and we consider $x:=0$ and $y := 0'$ in the identity $(42F25)$ we obtain that $0' = 0$. Then apply Lemma \ref{Lemma_0_C_equal_0}.\\
Case 9: $\mathbf A \models (42G15)$.   Then, setting $x:=0'$ and $y := 0$ in the identity $(42G15)$, it is easy to obtain that $0' = 0$. Then apply Lemma \ref{Lemma_0_C_equal_0}.
%
%
%
%
\end{proof}


\begin{Theorem} \label{031116_Theo1}
	If $\mathbf A \models 0 \to x \approx x$, then $\mathbf A \models (y)$ for all $(y) \in I$.
\end{Theorem}

\begin{proof}
 Observe that, by Lemma \ref{lemma_5terms}, 
\begin{equation} \label{231116_01}
\mathbf A \models t_1 \approx t_2 \approx t_4 \approx t_5.
\end{equation}
Also, by Lemma \ref{lemma_091116_01}, $\mathbf A \models (42A24).$ Hence, by Lemma \ref{lemma_identities_S_onevariable} (\ref{311016_05}), $\mathbf A \models (42A45).$  Now,  
 $\mathbf A \models (42B13)$ is true, since
\noindent $(a \to a) \to (b \to a) $
$\overset{   \ref{properties_of_I20_MC} (\ref{310516_01}) 
}{=}  b \to ((a \to a) \to a) $
$\overset{   \ref{031116_01_Assoc1} 
}{=}  b \to (a \to (a \to a)) $
$\overset{   \ref{properties_of_I20_MC} (\ref{310516_01}) 
}{=}  a \to (b \to (a \to a)) $
$\overset{   \ref{properties_of_I20_MC} (\ref{310516_01})
}{=}  a \to (a \to (b \to a)) $.
Observe that
\noindent $((a \to a) \to b) \to a $
$\overset{  
}{=}  t_1(a,a,b,a) $
$\overset{   (\ref{231116_01}) 
}{=}  t_2(a,a,b,a) $
$\overset{  
}{=}  b \to ((a \to a) \to a) $
$\overset{   \ref{031116_01_Assoc1} 
}{=}  b \to (a \to (a \to a)) $
$\overset{   \ref{properties_of_I20_MC} (\ref{310516_01}) 
}{=}  a \to (b \to (a \to a)) $
$\overset{   \ref{properties_of_I20_MC} (\ref{310516_01})
}{=}  a \to (a \to (b \to a)) $.
Consequently, $\mathbf A \models (42B15)$. From $(a \to a) \to (b \to a) =  t_5(b,a,a,a) = t_2(b,a,a,a) =
 a \to ((a \to b) \to a)$, it follows that $\mathbf A \models (42B23)$. Since $a \to ((a \to b) \to a) = t_2(b,a,a,a) = t_4(b,a,a,a) = ((a \to a) \to b) \to a$ we conclude that $\mathbf A \models (42B25)$.

Since $\mathbf A \models 0 \to x \approx x$, by Lemma \ref{lemma_070616_01}, 
	\begin{equation} \label{031116_04}
	\mathbf A \models (x \to y)' \approx x' \to y'.
	\end{equation}
Since
\noindent $(a \to (b \to a)) \to a $
$\overset{   \ref{properties_of_I20_MC} (\ref{310516_02}) 
}{=}  a' \to (a \to (b \to a))' $
$\overset{   \ref{properties_of_I20_MC} (\ref{310516_01}) 
}{=}  a' \to (b \to (a \to a))' $
$\overset{  (\ref{031116_04}) 
}{=}  a' \to (b' \to (a \to a)') $
$\overset{  (\ref{031116_04}) 
}{=}  a' \to (b' \to (a' \to a')) $
$\overset{   \ref{properties_of_I20_MC} (\ref{310516_02}) 
}{=}  a' \to (b' \to (a \to a)) $
$\overset{   \ref{properties_of_I20_MC} (\ref{310516_01}) 
}{=}  b' \to (a' \to (a \to a)) $
$\overset{   \ref{general_properties2} (\ref{281114_01}) 
}{=}  b' \to (a \to a) $
$\overset{   \ref{general_properties2} (\ref{031114_04}) 
}{=}  b' \to (a \to (a \to a)) $
$\overset{   \ref{properties_of_I20_MC} (\ref{310516_01}) 
}{=}  a \to (b' \to (a \to a)) $
$\overset{   \ref{properties_of_I20_MC} (\ref{310516_02}) 
}{=}  a \to (b' \to (a' \to a')) $
$\overset{   \ref{properties_of_I20_MC} (\ref{310516_01}) 
}{=}  a \to (a' \to (b' \to a')) $
$\overset{  (\ref{031116_04}) 
}{=}  a \to (a' \to (b \to a)') $
$\overset{   \ref{properties_of_I20_MC} (\ref{310516_02})
}{=}  a \to ((b \to a) \to a) $.
$\mathbf A$ satisfies the identity $(42D24)$. Notice that $\mathbf A$ satisfies the identity $(42D45)$, in view of
\noindent $((a \to b) \to a) \to a $
$\overset{  
}{=}  t_1(a,b,a,a) $
$\overset{   (\ref{231116_01}) 
}{=}  t_2(a,b,a,a) $
$\overset{  
}{=}  a \to ((b \to a) \to a) $
$\overset{ using (42D24) 
}{=}  a \to ((b \to a) \to a) $.
Since $((a \to b) \to b) \to a = t_4(b,b,a,a) = t_2(b,b,a,a) =  a \to ((b \to b) \to a)$, the algebra $\mathbf A \models (42F25)$. Also, since
\noindent $((a \to b) \to b) \to b $
$\overset{  
}{=}  t_4(b,b,a,b) $
$\overset{   (\ref{231116_01}) 
}{=}  t_2(b,b,a,b) $
$\overset{  
}{=}  a \to ((b \to b) \to b) $
$\overset{   \ref{031116_01_Assoc1} 
}{=}  a \to (b \to (b \to b)) $,
we have that $\mathbf A \models (42G15)$.
\end{proof}

Using Lemma \ref{lema_1_identity} and Theorem \ref{031116_Theo1}, one can easily verify the following Theorem.

\begin{Theorem} \label{theorem_081116_01}
$42\mathcal A24 =$ $42\mathcal A45 =$ $42\mathcal B13 =$ $42\mathcal B15=$ $42\mathcal B23 =$ $42\mathcal B25 =$ $42\mathcal D24 =$ $42\mathcal D45 =$ $42\mathcal F25 = $ $42\mathcal G15$.
\end{Theorem}


\begin{Lemma} \label{Old_Lemma_10.3(2)}
$42\mathcal A24 = 31\mathcal A12$. 
\end{Lemma}

\begin{proof}
Let $\mathbf A \in 31\mathcal A12$. Since, replacing $x:=0$ in the identity (31A12), we obtain $0 = 0'$ then $a = a'' = a' \to 0 = a' \to 0' = 0 \to a$ by Lemma \ref{general_properties} (\ref{cuasiConmutativeOfImplic}). Then, by Lemma \ref{lemma_091116_01},  $\mathbf A \in 42\mathcal A24$. Therefore $31\mathcal A12 \subseteq 42\mathcal A24$.

For the converse let consider $\mathbf A \in 42\mathcal A24$. Let $a \in A$. By Lemma \ref{lema_1_identity}
\begin{equation} \label{091116_02}
\mathbf A \models 0 \to x \approx x.
\end{equation}
Hence,
\noindent $(a \to a) \to a $
$\overset{   \ref{general_properties2} (\ref{291014_10}) 
}{=}  (0 \to a) \to a $
$\overset{(\ref{091116_02}) 
}{=}  a \to a $
$\overset{   \ref{general_properties2} (\ref{031114_04})
}{=}  a \to (a \to a) $.
\end{proof}

Let $J$ be the set consisting of the following identities: $(42B12)$, $(42B35)$, $(42C23)$, $(42C25)$, $(42C35)$ and $(42G25)$.

\begin{Lemma} \label{lema_1_identityB}
	Let $(x) \in J$. If $\mathbf A \models (x)$ then $A \models 0 \to (x \to x) \approx x \to x$.
\end{Lemma}

\begin{proof} Let $a \in A$.  Consider the following cases: 
\\
Case 1:  $\mathbf A \models (42B12)$.  
Then  
\noindent $a \to a $
$\overset{   \ref{general_properties_equiv} (\ref{LeftImplicationwithtilde}) 
}{=}  a \to (a' \to a) $
$\overset{  
}{=}  a \to ((a \to 0) \to a) $
$\overset{  (42B12) 
}{=}  a \to (a \to (0 \to a)) $
$\overset{   \ref{properties_of_I20_MC} (\ref{310516_01}) 
}{=}  a \to (0 \to (a \to a)) $
$\overset{   \ref{properties_of_I20_MC} (\ref{310516_01}) 
}{=}  0 \to (a \to (a \to a)) $
$\overset{   \ref{general_properties2} (\ref{031114_04}) 
}{=}  0 \to (a \to a) $.

\noindent Case 2:  $(x) = (42B35)$. 
Hence
\noindent $0 \to (a \to a) $
$\overset{   \ref{general_properties2} (\ref{071114_04}) and (\ref{311014_06}) 
}{=}  (0 \to a) \to (0 \to a) $
$\overset{   \ref{properties_of_I20_MC} (\ref{310516_01}) 
}{=}  0 \to ((0 \to a) \to a) $
$\overset{   \ref{general_properties2} (\ref{291014_10}) 
}{=}  0 \to ((a \to a) \to a) $
$\overset{   \ref{properties_of_I20_MC} (\ref{310516_01}) 
}{=}  (a \to a) \to (0 \to a) $
$\overset{  (42B35) 
}{=}  ((a \to a) \to 0) \to a $
$\overset{  
}{=}  (a \to a)' \to a $
$\overset{   \ref{properties_of_I20_MC} (\ref{310516_02}) 
}{=}  a' \to (a \to a) $
$\overset{   \ref{general_properties2} (\ref{281114_01}) 
}{=}  a \to a $.

\noindent Case 3:  $(x) = (42C23)$.  We get   
\noindent $a \to a $
$\overset{   \ref{general_properties_equiv} (\ref{TXX}) 
}{=} 0' \to (a \to a)  $
$\overset{  
}{=}  (0 \to 0) \to (a \to a) $
$\overset{  (42C23) 
}{=}  0 \to ((0 \to a) \to a) $
$\overset{   \ref{properties_of_I20_MC} (\ref{310516_01}) 
}{=}  (0 \to a) \to (0 \to a) $
$\overset{   \ref{general_properties2} (\ref{071114_04}) and (\ref{311014_06}) 
}{=}  0 \to (a \to a) $.

\noindent Case 4:  $(x) = (42C25)$.  One has  
\noindent $a \to a $
$\overset{   \ref{general_properties_equiv} (\ref{TXX}) 
}{=}  (0' \to a) \to a $
$\overset{  
}{=}  ((0 \to 0) \to a) \to a $
$\overset{  (42C25) 
}{=}  0 \to ((0 \to a) \to a) $
$\overset{   \ref{properties_of_I20_MC} (\ref{310516_01}) 
}{=}  (0 \to a) \to (0 \to a) $
$\overset{   \ref{general_properties2} (\ref{071114_04}) and (\ref{311014_06}) 
}{=}  0 \to (a \to a) $.

\noindent Case 5:  $(x) = (42C35)$.  Then  
\noindent $0 \to (a \to a)' $
$\overset{   \ref{properties_of_I20_MC} (\ref{310516_02}) 
}{=}  (a \to a) \to 0' $
$\overset{  
}{=}  (a \to a) \to (0 \to 0) $
$\overset{  (42C35) 
}{=}  ((a \to a) \to 0) \to 0 $
$\overset{ x \approx x''
}{=}  a \to a $.
Therefore,
\begin{equation} \label{071116_01}
\mathbf A \models 0 \to (x \to x)' \approx x \to x.
\end{equation}
Hence
\noindent $a \to a $
$\overset{  (\ref{071116_01}) 
}{=}  0 \to (a \to a)' $
$\overset{  (\ref{071116_01}) 
}{=}  0 \to (0 \to (a \to a)')' $
$\overset{   \ref{general_properties2} (\ref{031114_07}) 
}{=}  0 \to (a \to a)'' $
$\overset{ x \approx x'' 
}{=}  0 \to (a \to a) $.

\noindent Case 6:  $(x) = (42G25)$.  We have  
\noindent $0 \to (a \to a) $
$\overset{   \ref{general_properties2} (\ref{071114_04}) and (\ref{311014_06}) 
}{=}  (0 \to a) \to (0 \to a) $
$\overset{   \ref{properties_of_I20_MC} (\ref{310516_01}) 
}{=}  0 \to ((0 \to a) \to a) $
$\overset{   \ref{general_properties2} (\ref{291014_10}) 
}{=}  0 \to ((a \to a) \to a) $
$\overset{  (42G25) 
}{=}  ((0 \to a) \to a) \to a $
$\overset{   \ref{general_properties2} (\ref{291014_08})
}{=}  a \to a $,
which proves the lemma. 
\end{proof}

\begin{Theorem} \label{031116_Theo1B}
	If $\mathbf A \models 0 \to (x \to x) \approx x \to x$ then $\mathbf A \models (y)$, for all $(y) \in J$.
\end{Theorem}

\begin{proof} Let $a,b \in A$.
By Lemma \ref{lemma_071116_A} (\ref{081116_02}) and (\ref{081116_03}),
\begin{equation} \label{071116_03}
\mathbf A \models (x \to x) \to y' \approx ((x \to x) \to y)'
\end{equation}
and
\begin{equation} \label{071116_04}
\mathbf A \models (x \to x) \to (y \to z) \approx ((x \to x) \to y) \to z
\end{equation}
Using (\ref{071116_04}) we have that $\mathbf A$ satisfies (42B35) and (42C35).

\noindent Since
\noindent $a \to ((a \to b) \to a) $
$\overset{   \ref{properties_of_I20_MC} (\ref{310516_01}) 
}{=}  (a \to b) \to (a \to a) $
$\overset{  (I) 
}{=}  (((a \to a)' \to a) \to (b \to (a \to a))')' $
$\overset{  
}{=}  ((((a \to a) \to 0) \to a) \to (b \to (a \to a))')' $
$\overset{  (\ref{071116_04}) 
}{=}  (((a \to a) \to (0 \to a)) \to (b \to (a \to a))')' $
$\overset{   \ref{properties_of_I20_MC} (\ref{310516_01}) 
}{=}  ((0 \to ((a \to a) \to a)) \to (b \to (a \to a))')' $
$\overset{   \ref{general_properties2} (\ref{071114_04}) and (\ref{311014_06}) 
}{=}  ((( 0 \to (a \to a)) \to (0 \to a)) \to (b \to (a \to a))')' $
$\overset{  hyp  
}{=}  (((0 \to a) \to (0 \to a)) \to (b \to (a \to a))')' $
$\overset{  (\ref{071116_03}) 
}{=}  ((( 0 \to a) \to (0 \to a)) \to (b \to (a \to a)))'' $
$\overset{ x \approx x''  
}{=}  ((0 \to a) \to (0 \to a)) \to (b \to (a \to a)) $
$\overset{   \ref{general_properties2} (\ref{071114_04}) and (\ref{311014_06}) 
}{=}  ( 0 \to (a \to a)) \to (b \to (a \to a)) $
$\overset{  hyp
}{=}  ( 0 \to a) \to (b \to (a \to a)) $
$\overset{   \ref{properties_of_I20_MC} (\ref{310516_01}) 
}{=}  b \to ((0 \to a) \to (a \to a)) $
$\overset{   \ref{properties_of_I20_MC} (\ref{310516_01}) 
}{=}  b \to (a \to ((0 \to a) \to a)) $
$\overset{   \ref{general_properties2} (\ref{291014_10}) 
}{=}  b \to (a \to ((a \to a) \to a)) $
$\overset{   \ref{properties_of_I20_MC} (\ref{310516_01}) 
}{=} b \to ((a \to a) \to (a \to a))  $
$\overset{   \ref{general_properties2} (\ref{250315_05}) 
}{=}  b \to (a \to a) $
$\overset{   \ref{general_properties2} (\ref{031114_04}) 
}{=}  b \to (a \to (a \to a)) $
$\overset{   \ref{properties_of_I20_MC} (\ref{310516_01}) 
}{=}  a \to (b \to (a \to a)) $
$\overset{   \ref{properties_of_I20_MC} (\ref{310516_01}) 
}{=}  a \to (a \to (b \to a)) $,
the algebra $\mathbf A$ satisfies (42B12).
By Lemma \ref{lemma_071116_A} (\ref{081116_04}) and Lemma \ref{properties_of_I20_MC} (\ref{310516_01}), $\mathbf A$ satisfies (42C23).
In view of (\ref{071116_04}) and (42C23), we have $\mathbf A \models (x \to x) \to (y \to y) \approx x \to ((x \to y) \to y)$ proving that the identity (42C25) holds in $\mathbf A$.

It remains to verify that the algebra $\mathbf A$ holds the identity (42G25). To finish off the proof,
\noindent $((a \to b) \to b) \to b $
$\overset{  (I) 
}{=}  ((b' \to a) \to (b \to b)')' \to b $
$\overset{   \ref{properties_of_I20_MC} (\ref{310516_02}) 
}{=}  b' \to ((b' \to a) \to (b \to b)')'' $
$\overset{  x'' \approx x 
}{=}  b' \to ((b' \to a) \to (b \to b)') $
$\overset{   \ref{properties_of_I20_MC} (\ref{310516_02}) 
}{=}  b' \to ((b \to b) \to (b' \to a)')  $
$\overset{   \ref{properties_of_I20_MC} (\ref{310516_01}) 
}{=}  (b \to b) \to (b' \to (b' \to a)') $
$\overset{   \ref{properties_of_I20_MC} (\ref{310516_02}) 
}{=}  (b \to b) \to ((b' \to a) \to b) $
$\overset{  (I) 
}{=}  (b \to b) \to ((b' \to b') \to (a \to b)')' $
$\overset{   \ref{properties_of_I20_MC} (\ref{310516_02}) 
}{=}  (b \to b) \to ((b \to b) \to (a \to b)')' $
$\overset{  (\ref{071116_03}) 
}{=}  (b \to b) \to ((b \to b) \to (a \to b))'' $
$\overset{  x'' \approx x 
}{=}  (b \to b) \to ((b \to b) \to (a \to b)) $
$\overset{   \ref{general_properties2} (\ref{031114_04}) 
}{=}  (b \to b) \to (a \to b) $
$\overset{   \ref{properties_of_I20_MC} (\ref{310516_01}) 
}{=}  a \to ((b \to b) \to b) $.
Consequently,
$\mathbf A \models (42G25).$
\end{proof}

In view of Lemma \ref{lema_1_identityB} and Theorem \ref{031116_Theo1B}, the following theorem is immediate.
\begin{Theorem} \label{theo_081116_01}
 $42\mathcal B12 =$ $42\mathcal B35 =$ $42\mathcal C23 =$ $42\mathcal C25=$ $42\mathcal C35 =$ $42\mathcal G25$.	
\end{Theorem}

\begin{Lemma} \label{lemma_081116_06}
	$42 \mathcal A12 = 42 \mathcal D23 = 42 \mathcal D35$.
\end{Lemma}

\begin{proof}
	In view of Lemma \ref{lemma_identities_S_onevariable} (\ref{311016_06}), we have  

$$\mathbf A \models x \to ((y \to x) \to x) \approx ((x \to y) \to x) \to x.$$ Hence, 
$$\mathbf A \models (42D23) \mbox{  if and only if } \mathbf A \models (42D35), $$
proving that $\mathcal D23 = 42 \mathcal D35$.

\noindent	Assume $\mathbf A \in 42\mathcal A12$ and $a,b \in A$. Observe that 
\noindent $((a \to b) \to a) \to a $
$\overset{   \ref{general_properties2} (\ref{291014_09}) 
}{=}  (a \to (b \to a)')' \to a $
$\overset{   \ref{properties_of_I20_MC} (\ref{310516_01}) 
}{=}  a' \to (a \to (b \to a)') $
$\overset{   \ref{properties_of_I20_MC} (\ref{310516_02})  
}{=}  a' \to ((b \to a) \to a') $
$\overset{   \ref{properties_of_I20_MC} (\ref{310516_01})  
}{=}  (b \to a) \to (a' \to a') $
$\overset{   \ref{properties_of_I20_MC} (\ref{310516_02}) 
}{=}  (b \to a) \to (a \to a) $
$\overset{   \ref{properties_of_I20_MC} (\ref{310516_02}) 
}{=}  (a' \to b') \to (a \to a) $
$\overset{   \ref{general_properties2} (\ref{031114_04}) 
}{=}  (a' \to (a' \to b')) \to (a \to a) $
$\overset{   \ref{general_properties2} (\ref{031114_04}) 
}{=}  (a' \to (a' \to (a' \to b'))) \to (a \to a) $
$\overset{  (42A12) 
}{=}  (a' \to ((a' \to a') \to b)) \to (a \to a) $
$\overset{   \ref{properties_of_I20_MC} (\ref{310516_01}) 
}{=}  ((a' \to a') \to (a' \to b)) \to (a \to a) $
$\overset{   \ref{properties_of_I20_MC} (\ref{310516_02}) 
}{=}  ((a \to a) \to (a' \to b)) \to (a \to a) $
$\overset{   \ref{general_properties2} (\ref{291014_10}) 
}{=}  (0 \to (a' \to b)) \to (a \to a) $
$\overset{   \ref{properties_of_I20_MC} (\ref{310516_01}) 
}{=}  (a' \to (0 \to b)) \to (a \to a) $
$\overset{   \ref{properties_of_I20_MC} (\ref{310516_02}) 
}{=}  ((0 \to b)' \to a) \to (a \to a) $
$\overset{   \ref{general_properties2} (\ref{250615_04}) 
}{=}  (0 \to (0 \to b)'') \to (a \to a) $
$\overset{  x \approx x''  
}{=}  (0 \to (0 \to b)) \to (a \to a) $
$\overset{   \ref{general_properties2} (\ref{031114_04}) 
}{=}  (0 \to b) \to (a \to a) $
$\overset{   \ref{properties_of_I20_MC} (\ref{310516_02}) 
}{=}  (0 \to b) \to (a' \to a') $
$\overset{    \ref{general_properties2} (\ref{250615_04}) 
}{=}  (b' \to a') \to (a' \to a') $
$\overset{   \ref{properties_of_I20_MC} (\ref{310516_02}) 
}{=}  (a \to b) \to (a \to a) $.
Hence $\mathbf A \in 42\mathcal D35$.
Now assume that $\mathbf A \in 42\mathcal D35$ and $a,b \in A$. \\
Since
\noindent $0 \to a $
$\overset{  x \approx x'' 
}{=}  (0 \to a)'' $
$\overset{  
}{=}  ((0 \to a) \to 0) \to 0 $
$\overset{  (42D35) 
}{=}  (0 \to a) \to (0 \to 0) $
$\overset{  
}{=}  (0 \to a) \to 0' $
$\overset{   \ref{properties_of_I20_MC} (\ref{310516_02}) 
}{=}  0 \to (0 \to a)' $
$\overset{   \ref{general_properties2} (\ref{031114_07}) 
}{=}  0 \to a' $,
the algebra $\mathbf A$ satisfies
\begin{equation} \label{021116_04}
0 \to x \approx 0 \to x'.
\end{equation}
Therefore, 
\noindent $a \to ((a \to a) \to b) $
$\overset{   \ref{properties_of_I20_MC} (\ref{310516_01}) 
}{=}  (a \to a) \to (a \to b) $
$\overset{   \ref{general_properties2} (\ref{250615_04}) 
}{=}  (0 \to a') \to (a \to b) $
$\overset{   \ref{properties_of_I20_MC} (\ref{310516_01}) 
}{=}  a \to ((0 \to a') \to b) $
$\overset{  (\ref{021116_04}) 
}{=}  a \to ((0 \to a) \to b) $
$\overset{   \ref{general_properties2} (\ref{291014_10})  
}{=}  a \to ((b \to a) \to b) $
$\overset{   \ref{properties_of_I20_MC} (\ref{310516_01}) 
}{=}  (b \to a) \to (a \to b) $
$\overset{   \ref{general_properties2} (\ref{250315_05})  
}{=}  a \to b $
$\overset{   \ref{general_properties2} (\ref{031114_04}) 
}{=}  a \to (a \to b) $
$\overset{   \ref{general_properties2} (\ref{031114_04}) 
}{=}  a \to (a \to (a \to b)) $.
Hence $\mathbf A \in 42\mathcal A12$.
\end{proof}

Let $K$ be the set consisting of the following identities: $(42C12)$, $(42C13)$, $(42C15)$, $(42D12)$ and $(42F35)$.

\begin{Lemma} \label{lema_1_identityC}
	Let $(x) \in K$. If $\mathbf A \models (x)$ then $\mathbf A \models 0 \to x \approx x \to x$.
\end{Lemma}

\begin{proof} Let $a \in A$.
	
If $\mathbf A \models (42C12)$,
\noindent $a \to a $
$\overset{   \ref{properties_of_I20_MC} (\ref{310516_02}) 
}{=}  a' \to a' $
$\overset{  x'' \approx x 
}{=}  a' \to a''' $
$\overset{  
}{=}  a' \to ((a' \to 0) \to 0) $
$\overset{  (42C12) 
}{=}  a' \to (a' \to (0 \to 0)) $
$\overset{  
}{=}  a' \to (a' \to 0') $
$\overset{   \ref{general_properties2} (\ref{031114_04}) 
}{=}  a' \to 0' $
$\overset{   \ref{general_properties} (\ref{cuasiConmutativeOfImplic2})
}{=}  0 \to a $.\\
If $\mathbf A \models (42C13)$,
\noindent $0 \to a $
$\overset{   \ref{properties_of_I20_MC} (\ref{310516_02}) 
}{=}  a' \to 0' $
$\overset{   \ref{general_properties2} (\ref{031114_04}) 
}{=}  a' \to (a' \to 0') $
$\overset{  
}{=}  a' \to (a' \to (0 \to 0)) $
$\overset{  (42C13) 
}{=}  (a' \to a') \to (0 \to 0) $
$\overset{  
}{=}  (a' \to a') \to 0' $
$\overset{   \ref{general_properties} (\ref{cuasiConmutativeOfImplic}) 
}{=}  0 \to (a' \to a')' $
$\overset{   \ref{general_properties2} (\ref{191114_05}) 
}{=}  0 \to (a'' \to a'') $
$\overset{  x'' \approx x 
}{=}  0 \to (a \to a) $
$\overset{   \ref{general_properties2} (\ref{031114_04}) 
}{=}  0 \to (0 \to (a \to a)) $
$\overset{  (42C13) 
}{=}  (0 \to 0) \to (a \to a) $
$\overset{  
}{=}  0' \to (a \to a) $
$\overset{   \ref{general_properties_equiv} (\ref{TXX}) 
}{=}  a \to a $.\\
If $\mathbf A \models (42C15)$,
\noindent $0 \to a $
$\overset{   \ref{properties_of_I20_MC} (\ref{310516_02}) 
}{=}  a' \to 0' $
$\overset{   \ref{general_properties2} (\ref{031114_04}) 
}{=}  a' \to (a' \to 0') $
$\overset{  
}{=}  a' \to (a' \to (0 \to 0)) $
$\overset{  (42C15) 
}{=}  ((a' \to a') \to 0) \to 0 $
$\overset{  x'' \approx x 
}{=}  a' \to a' $
$\overset{   \ref{properties_of_I20_MC} (\ref{310516_02}) 
}{=}  a \to a $.\\
If $\mathbf A \models (42D12)$,
\noindent $0 \to a $
$\overset{   \ref{properties_of_I20_MC} (\ref{310516_02}) 
}{=}  a' \to 0' $
$\overset{   \ref{general_properties_equiv} (\ref{LeftImplicationwithtilde}) 
}{=}  a' \to (0 \to 0') $
$\overset{  
}{=}  a' \to (0 \to (0 \to 0)) $
$\overset{   \ref{properties_of_I20_MC} (\ref{310516_01}) 
}{=}  0 \to (a' \to (0 \to 0)) $
$\overset{  (42D12) 
}{=}  0 \to ((a' \to 0) \to 0) $
$\overset{  x'' \approx x 
}{=}  0 \to (a \to 0) $
$\overset{  
}{=}  0 \to a' $.\\
Hence,
\begin{equation} \label{081116_01}
\mathbf A \models 0 \to x \approx 0 \to x'.
\end{equation}
Therefore,
\noindent $a \to a $
$\overset{   \ref{general_properties_equiv} (\ref{LeftImplicationwithtilde}) 
}{=}  a \to (a' \to a) $
$\overset{   \ref{general_properties_equiv} (\ref{LeftImplicationwithtilde}) 
}{=}  a \to ((a \to a') \to a) $
$\overset{   \ref{general_properties2} (\ref{291014_10}) 
}{=}  a \to ((0 \to a') \to a) $
$\overset{  (\ref{081116_01}) 
}{=}  a \to ((0 \to a) \to a) $
$\overset{  (42D12) 
}{=}  a \to (0 \to (a \to a)) $
$\overset{   \ref{properties_of_I20_MC} (\ref{310516_01}) 
}{=}  0 \to (a \to (a \to a)) $
$\overset{   \ref{general_properties2} (\ref{031114_04}) 
}{=}  0 \to (a \to a) $
$\overset{   \ref{general_properties2} (\ref{071114_04}) and (\ref{311014_06}) 
}{=}  (0 \to a) \to (0 \to a) $
$\overset{  (\ref{081116_01}) 
}{=}  (0 \to a') \to (0 \to a) $
$\overset{   \ref{general_properties2} (\ref{071114_04}) and (\ref{311014_06}) 
}{=}  0 \to (a' \to a) $
$\overset{  
}{=}  0 \to a. $\\
If $\mathbf A \models (42F35)$,
\noindent $0 \to a $
$\overset{   \ref{general_properties2} (\ref{281114_01}) 
}{=}  a' \to (0 \to a) $
$\overset{  
}{=}  (a \to 0) \to (0 \to a) $
$\overset{  (42F35) 
}{=}  [(a \to 0) \to 0] \to a $
$\overset{  
}{=}  a'' \to a $
$\overset{  x'' \approx x.
}{=}  a \to a $.
\end{proof}

\begin{Theorem} \label{081116_Theo1C} 
	If $\mathbf A \models 0 \to x \approx x \to x$ then $\mathbf A \models (y)$ for all $(y) \in K$.
\end{Theorem}

\begin{proof} Let $a,b \in A$. 
 By the hypothesis and Lemma \ref{lemma_081116_A} (\ref{231116_03}) 
$\mathbf A$ satisfies
\begin{equation} \label{condition_0implximplx_ximplx}
0 \to (x \to x) \approx x \to x.
\end{equation}
Hence, by Lemma \ref{lemma_071116_A} (\ref{081116_04}) and Lemma \ref{properties_of_I20_MC} (\ref{310516_01}), the algebra $\mathbf A$ satisfies (42C23).



In view of Theorem \ref{theo_081116_01}, the algebra $\mathbf A$ satisfies the identity (42B12) too.\\
Since
\noindent $(a \to a) \to (b \to b) $
$\overset{  hyp 
}{=}  (0 \to a) \to (0 \to b) $
$\overset{   \ref{general_properties2} (\ref{071114_04}) and (\ref{311014_06}) 
}{=}  0 \to (a \to b) $
$\overset{   \ref{properties_of_I20_MC} (\ref{310516_01}) 
}{=}  a \to (0 \to b) $
$\overset{  hyp 
}{=}  a \to (b \to b) $
$\overset{   \ref{general_properties2} (\ref{031114_04})
}{=}  a \to (a \to (b \to b)) $,
the algebra $\mathbf A$ satisfies (42C13). 
Therefore,
\noindent $a \to ((a \to b) \to b)$
$\overset{  (42C23) 
}{=}  (a \to a) \to (b \to b) $
$\overset{  (42C13) 
}{=}  a \to (a \to (b \to b))  $.
Consequently, $\mathbf A \models (42C12)$. 
By Lemma \ref{lemma_071116_A} (\ref{081116_03}),
\begin{equation} \label{231116_05}
\mathbf A \models (x \to x) \to (y \to z) \approx ((x \to x) \to y) \to z.
\end{equation}	
	Therefore,
\noindent $a \to (a \to (b \to b)) $
$\overset{  (42C12) 
}{=}  a \to ((a \to b) \to b) $
$\overset{   \ref{properties_of_I20_MC} (\ref{310516_01}) 
}{=}  (a \to b) \to (a \to b) $
$\overset{   \ref{lemma_071116_A} (\ref{081116_04}) and (\ref{condition_0implximplx_ximplx}) 
}{=}  (a \to a) \to (b \to b) $
$\overset{  (\ref{231116_05}) 
}{=}  ((a \to a) \to b) \to b $	
proving (42C15).
By Lemma \ref{lemma_081116_A} (\ref{081116_06}) $\mathbf A$ satisfies
\begin{equation} \label{condition_0implxC_0implx}
0 \to x' \approx 0 \to x . 
\end{equation}

 Since
\noindent $a \to ((b \to a) \to a) $
$\overset{   \ref{properties_of_I20_MC} (\ref{310516_01}) 
}{=}  (b \to a) \to (a \to a)  $
$\overset{ hyp 
}{=}  (b \to a) \to (0 \to a) $
$\overset{  (I) 
}{=}  (((0 \to a)' \to b) \to (a \to (0 \to a))')' $
$\overset{   \ref{properties_of_I20_MC} (\ref{310516_01}) 
}{=}  (((0 \to a)' \to b) \to (0 \to (a \to a))')' $
$\overset{  (\ref{condition_0implximplx_ximplx}) 
}{=}  (((0 \to a)' \to b) \to (a \to a)')' $
$\overset{ hyp 
}{=}  (((0 \to a)' \to b) \to (0 \to a)')' $
$\overset{   \ref{general_properties_equiv} (\ref{TXX}) 
}{=}  (((0 \to a)' \to b) \to (0' \to (0 \to a))')' $
$\overset{  (I) 
}{=}  (b \to 0') \to (0 \to a) $
$\overset{   \ref{general_properties} (\ref{cuasiConmutativeOfImplic}) 
}{=}  (0 \to b') \to (0 \to a) $
$\overset{  (\ref{condition_0implxC_0implx}) 
}{=}  (0 \to b) \to (0 \to a) $
$\overset{ hyp 
}{=}  (0 \to b) \to (a \to a) $
$\overset{   \ref{properties_of_I20_MC} (\ref{310516_01}) 
}{=}  a \to ((0 \to b) \to a) $
$\overset{   \ref{general_properties2} (\ref{291014_10}) 
}{=}  a \to ((a \to b) \to a) $
$\overset{  (42B12) 
}{=}  a \to (a \to (b \to a)) $
$\overset{   \ref{properties_of_I20_MC} (\ref{310516_01}) 
}{=}  a \to (b \to (a \to a)) $,
the identity (42D12) holds in  $\mathbf A$.
Finally, $\mathbf A \models (42F35)$, in view of the following calculus:
\noindent $((a \to b) \to b) \to a $
$\overset{  (I) 
}{=}  ((a' \to (a \to b)) \to (b \to a)')' $
$\overset{   \ref{properties_of_I20_MC} (\ref{310516_01}) 
}{=}  ((a \to (a' \to b)) \to (b \to a)')' $
$\overset{   \ref{properties_of_I20_MC} (\ref{310516_02}) 
}{=}  ((a \to (b' \to a)) \to (b \to a)')' $
$\overset{   \ref{properties_of_I20_MC} (\ref{310516_01}) 
}{=}  ((b' \to (a \to a)) \to (b \to a)')' $
$\overset{ hyp 
}{=}  ((b' \to (0 \to a)) \to (b \to a)')' $
$\overset{   \ref{properties_of_I20_MC} (\ref{310516_01}) 
}{=}  ((0 \to (b' \to a)) \to (b \to a)')' $
$\overset{   \ref{general_properties2} (\ref{071114_04}) and (\ref{311014_06}) 
}{=}  (((0 \to b') \to (0 \to a)) \to (b \to a)')' $
$\overset{  (\ref{condition_0implxC_0implx}) 
}{=}  (((0 \to b) \to (0 \to a)) \to (b \to a)')' $
$\overset{   \ref{general_properties2} (\ref{071114_04}) and (\ref{311014_06}) 
}{=}  ((0 \to (b \to a)) \to (b \to a)')' $
$\overset{   \ref{general_properties2} (\ref{291014_10}) 
}{=}  (((b \to a)' \to (b \to a)) \to (b \to a)')' $
$\overset{   \ref{general_properties_equiv} (\ref{LeftImplicationwithtilde}) 
}{=}  ((b \to a) \to (b \to a)')' $
$\overset{   \ref{general_properties_equiv} (\ref{LeftImplicationwithtilde}) 
}{=}  (b \to a)'' $
$\overset{  x'' \approx x 
}{=}  b \to a $
$\overset{   \ref{general_properties2} (\ref{250315_05}) 
}{=}  (a \to b) \to (b \to a) $,
completing the proof.
\end{proof}



\begin{Theorem} \label{theo_081116_02}
	 $42\mathcal C12 =$ $42\mathcal C13 =$ $42\mathcal C15 =$ $42\mathcal D12=$ $42\mathcal F35$.
\end{Theorem}

\begin{proof}
If $\mathbf A \in 42\mathcal C12$ then, by Lemma \ref{lema_1_identityC}, $\mathbf A \models 0 \to x \approx x \to x$. Hence, using Theorem \ref{081116_Theo1C}, $\mathbf A \models (y)$ for all $(y) \in K$. Consequently $\mathbf A \models (42C13)$, proving $\mathbf A \in 42\mathcal C12 \subseteq\mathbf A \in 42\mathcal C13$. Similar arguments can be used for the remaining inclusions.  
\end{proof}

\section{Weak associative laws with length 4 and with 3 variables.} \label{section_length4_3variables}

Possible words of length 4 with 3 variables are: 
\medskip  

\begin{minipage}{0.8 \textwidth}
\begin{multicols}{2}
A: $\langle x, x, y, z\rangle$, 

B: $\langle x, y, x, z\rangle$, 

C: $\langle x, y, y, z\rangle$, 

D: $\langle x, y, z, x\rangle$, 

E: $\langle x, y, z, y\rangle$, 

F: $\langle x, y, z, z\rangle$. 
\end{multicols}
\end{minipage}

\medskip \noindent Ways in which a word of size 4 can be bracketed:

\begin{minipage}{0.8 \textwidth}
	\begin{multicols}{2}
1: $a \to (a \to (a \to a))$, 

2: $a \to ((a \to a) \to a)$, 

3: $(a \to a) \to (a \to a)$, 

4: $(a \to (a \to a)) \to a$, 

5: $((a \to a) \to a) \to a$. 
\end{multicols}
\end{minipage}

\medskip
Note that these identities are precisely the identities of Bol-Moufang type which were, as mentioned in the Introduction, analysed in 
\cite{cornejo2016BolMoufang}  with a slightly different notation wherein the first two digits were not used; for example,  $43\mathcal{A} 23$ was denoted by $\mathcal{A}_{23}$, etc.

%
	
\begin{Theorem} {\bf \cite{cornejo2016BolMoufang}} \label{theo_4_dist_var}
	There are 4 nontrivial varieties of Bol-Moufang type that are distinct from each other:
	$\mathcal{SL}$, $43\mathcal A12$,   $43\mathcal A23$ and  $43\mathcal F25$; and they satisfy the following inclusions: 
	\begin{enumerate}
		\item $\mathcal{SL} \subset 43\mathcal A23 \subset 43\mathcal F25$, 
		\item $\mathcal{SL} \subset 43\mathcal A12$,
		\item $\mathcal{BA} \subset 43\mathcal A12 \subset 43\mathcal F25$, 
		\item $43\mathcal F25 \subset \mathcal{S}$,
		\item $\mathcal{SL} = 43\mathcal A23 \cap 43\mathcal A12$.  
	\end{enumerate}	
\end{Theorem}


%



\begin{Lemma} \label{Ola_Lemma_10.3(5)}
$43\mathcal F25 =  42\mathcal B35$.
\end{Lemma}

\begin{proof}
Let $\mathbf A \in 42\mathcal B35$.   
By Lemma \ref{lema_1_identityB},
\begin{equation} \label{111116_02}
\mathbf A \models 0 \to (x \to x) \approx x \to x.
\end{equation}
Hence, using Lemma \ref{lemma_071116_A}, we have
\begin{equation} \label{111116_03}
\mathbf A \models (x \to x) \to y' \approx ((x \to x) \to y)'.
\end{equation}
Let $a,b,c \in A$.  Then
\noindent $((a \to b) \to c) \to c $
$\overset{  (I) 
}{=}  ((c' \to a) \to (b \to c)')' \to c $
$\overset{   \ref{properties_of_I20_MC} (\ref{310516_02}) 
}{=}  c' \to ((c' \to a) \to (b \to c)')'' $
$\overset{  x'' \approx x 
}{=}  c' \to ((c' \to a) \to (b \to c)') $
$\overset{   \ref{properties_of_I20_MC} (\ref{310516_02}) 
}{=}  c' \to ((b \to c) \to (c' \to a)')  $
$\overset{   \ref{properties_of_I20_MC} (\ref{310516_01}) 
}{=}  (b \to c) \to (c' \to (c' \to a)') $
$\overset{   \ref{properties_of_I20_MC} (\ref{310516_02}) 
}{=}  (b \to c) \to ((c' \to a) \to c) $
$\overset{  (I) 
}{=}  (b \to c) \to ((c' \to c') \to (a \to c)')' $
$\overset{   \ref{properties_of_I20_MC} (\ref{310516_02}) 
}{=}  (b \to c) \to ((c \to c) \to (a \to c)')' $
$\overset{  (\ref{111116_03}) 
}{=}  (b \to c) \to ((c \to c) \to (a \to c))'' $
$\overset{  x'' \approx x 
}{=}  (b \to c) \to ((c \to c) \to (a \to c)) $
$\overset{   \ref{general_properties2} (\ref{070715_06}) 
}{=}  (b \to c) \to (a \to c) $
$\overset{   \ref{properties_of_I20_MC} (\ref{310516_01})
}{=}  a \to ((b \to c) \to c) $.
Hence, $\mathbf A \models (43F25)$ and consequently, $$42\mathcal B35 \subseteq  43\mathcal F25.$$       

Assume now that $\mathbf A \in 43\mathcal F25$ and $a \in A$. Then 
\noindent $0 \to (a \to a) $
$\overset{   \ref{general_properties_equiv} (\ref{TXX}) 
}{=}  0 \to ((0' \to a) \to a) $
$\overset{  (43F25) 
}{=}  ((0 \to 0') \to a) \to a $
$\overset{   \ref{general_properties_equiv} (\ref{LeftImplicationwithtilde}) \
}{=}  (0' \to a) \to a $
$\overset{   \ref{general_properties_equiv} (\ref{TXX}) 
}{=}  a \to a $.
Hence, using Lemma \ref{lemma_071116_A}, we get
\begin{equation} \label{111116_04}
\mathbf A \models (x \to x) \to (y \to z) \approx ((x \to x) \to y) \to z
\end{equation}
Using (\ref{111116_04})  we have that $\mathbf A \models (42B35)$.
This completes the proof.	
\end{proof}


\begin{Lemma} \label{lemma_061216_01}
	$43\mathcal A12 = 42\mathcal C12$.
\end{Lemma}

\begin{proof}
Let $\mathbf A \in 43\mathcal A12$	and $a \in A$. Observe that
\noindent$a \to a$
$\overset{\ref{properties_of_I20_MC} (\ref{310516_02})}{ =}  a' \to a' $  
$\overset{x \approx x'' }{ = } a' \to a''' $  
$\overset{}{ =}  a' \to ((a' \to 0) \to 0) $ 
$ \overset{(43A12)}{=}  a' \to (a' \to (0 \to 0))$ 
$\overset{\ref{general_properties2} (\ref{031114_04})} =  a' \to (0 \to 0)$  
$\overset{}{ =}  a' \to 0' $ 
$\overset{\ref{properties_of_I20_MC} (\ref{310516_02})} { =}  0 \to a $. 
Hence, by Theorem \ref{081116_Theo1C}, $\mathbf A \in 42\mathcal C12$, implying $43\mathcal A12 \subseteq 42\mathcal C12$.

Conversely, take $\mathbf A \in 42\mathcal C12$ and $a,b,c \in A$. By Lemma \ref{lema_1_identityC}, 
\begin{equation} \label{condition_0implx_ximplx_1411}
\mathbf A \models 0 \to x \approx x \to x.
\end{equation}

Hence, by Lemma \ref{lemma_071116_A},
\begin{equation}  \label{141116_01}
\mathbf A \models (x \to x) \to y' \approx ((x \to x) \to y)'
\end{equation}
and 
\begin{equation}  \label{141116_02}
\mathbf A \models (x \to x) \to (y \to z) \approx ((x \to x) \to y) \to z.
\end{equation}
Therefore,
\noindent $a \to ((a \to b) \to c)$
$\overset{  (I) 
}{=}  a \to ((c' \to a) \to (b \to c)')' $
$\overset{   \ref{properties_of_I20_MC} (\ref{310516_02}) 
}{=}  ((c' \to a) \to (b \to c)') \to a' $
$\overset{  (I) 
}{=}  \{(a'' \to (c' \to a)) \to ((b \to c)' \to a')'\}' $
$\overset{  x \approx x'' 
}{=}  \{(a \to (c' \to a)) \to ((b \to c)' \to a')'\}' $
$\overset{   \ref{properties_of_I20_MC} (\ref{310516_01}) 
}{=}  \{(c' \to (a \to a)) \to ((b \to c)' \to a')'\}' $
$\overset{   \ref{general_properties2} (\ref{031114_04}) 
}{=}  \{(c' \to (c' \to (a \to a))) \to ((b \to c)' \to a')'\}' $
$\overset{  (42C12) 
}{=}  \{(c' \to ((c' \to a) \to a)) \to ((b \to c)' \to a')'\}' $
$\overset{   \ref{properties_of_I20_MC} (\ref{310516_01}) 
}{=}  \{((c' \to a) \to (c' \to a)) \to ((b \to c)' \to a')'\}' $
$\overset{  (\ref{141116_01}) 
}{=}  \{((c' \to a) \to (c' \to a)) \to ((b \to c)' \to a')\}'' $
$\overset{  x \approx x'' 
}{=}  ((c' \to a) \to (c' \to a)) \to ((b \to c)' \to a') $
$\overset{  (\ref{141116_02}) 
}{=}  (((c' \to a) \to (c' \to a)) \to (b \to c)') \to a' $
$\overset{  (\ref{141116_01}) 
}{=}  (((c' \to a) \to (c' \to a)) \to (b \to c))' \to a' $
$\overset{   \ref{properties_of_I20_MC} (\ref{310516_01}) 
}{=}  ((c' \to ((c' \to a) \to a)) \to (b \to c))' \to a' $
$\overset{  (42C12) 
}{=}  ((c' \to (c' \to (a \to a))) \to (b \to c))' \to a' $
$\overset{   \ref{general_properties2} (\ref{031114_04}) 
}{=}  ((c' \to (a \to a)) \to (b \to c))' \to a' $
$\overset{  (\ref{condition_0implx_ximplx_1411}) 
}{=}  ((c' \to (0 \to a)) \to (b \to c))' \to a' $
$\overset{   \ref{properties_of_I20_MC} (\ref{310516_01}) 
}{=}  ((0 \to (c' \to a)) \to (b \to c))' \to a' $
$\overset{   \ref{general_properties2} (\ref{071114_04}) and (\ref{311014_06}) 
}{=}  (((0 \to c') \to (0 \to a)) \to (b \to c))' \to a' $
$\overset{  (\ref{condition_0implx_ximplx_1411}) 
}{=}  (((c' \to c') \to (0 \to a)) \to (b \to c))' \to a' $
$\overset{   \ref{properties_of_I20_MC} (\ref{310516_02}) 
}{=}  (((c \to c) \to (0 \to a)) \to (b \to c))' \to a'  $
$\overset{  (\ref{condition_0implx_ximplx_1411}) 
}{=}  (((0 \to c) \to (0 \to a)) \to (b \to c))' \to a' $
$\overset{   \ref{general_properties2} (\ref{071114_04}) and (\ref{311014_06})
}{=}  ((0 \to (c \to a)) \to (b \to c))' \to a' $.

So,
\noindent $a \to ((a \to b) \to c) $
$\overset{   \ref{properties_of_I20_MC} (\ref{310516_01}) 
}{=}  ((c \to (0 \to a)) \to (b \to c))' \to a' $
$\overset{   \ref{properties_of_I20_MC} (\ref{310516_02}) 
}{=}  (((0 \to a)' \to c') \to (b \to c))' \to a' $
$\overset{   \ref{properties_of_I20_MC} (\ref{310516_02}) 
}{=}  (((0 \to a)' \to c') \to (c' \to b'))' \to a' $
$\overset{   \ref{general_properties2} (\ref{250615_04}) 
}{=}  ((0 \to (0 \to a)'') \to (c' \to b'))' \to a' $
$\overset{  x \approx x'' 
}{=}  ((0 \to (0 \to a)) \to (c' \to b'))' \to a' $
$\overset{   \ref{general_properties2} (\ref{031114_04}) 
}{=}  ((0 \to a) \to (c' \to b'))' \to a' $
$\overset{   \ref{properties_of_I20_MC} (\ref{310516_02}) 
}{=}  ((0 \to a) \to (b \to c))' \to a' $
$\overset{   \ref{general_properties2} (\ref{281114_01}) 
}{=}  a \to \{((0 \to a) \to (b \to c))' \to a'\} $
$\overset{  (\ref{condition_0implx_ximplx_1411}) 
}{=}  a \to \{((a \to a) \to (b \to c))' \to a'\} $
$\overset{  (\ref{141116_01}) 
}{=}  a \to \{((a \to a) \to (b \to c)') \to a'\} $
$\overset{  (I) 
}{=}  a \to \{(a \to (a \to a)) \to ((b \to c)' \to a')'\}' $
$\overset{   \ref{general_properties2} (\ref{031114_04}) 
}{=}  a \to \{(a \to a) \to ((b \to c)' \to a')'\}' $
$\overset{   \ref{properties_of_I20_MC} (\ref{310516_02}) 
}{=}  a \to \{((b \to c)' \to a') \to (a \to a)'\}' $
$\overset{   \ref{properties_of_I20_MC} (\ref{310516_02}) 
}{=}  a \to \{(a \to (b \to c)) \to (a \to a)'\}' $
$\overset{   \ref{properties_of_I20_MC} (\ref{310516_02}) 
}{=}  a \to \{(a \to (b \to c)) \to (a' \to a')'\}' $
$\overset{  (\ref{condition_0implx_ximplx_1411}) 
}{=}  a \to \{(a \to (b \to c)) \to (0 \to a')'\}' $
$\overset{  (I)      
}{=}  a \to (((b \to c) \to 0) \to a' ) $
$\overset{   \ref{properties_of_I20_MC} (\ref{310516_02})
}{=}  a \to (a \to (b \to c)) $.
Hence $\mathbf A \in 43\mathcal A12$.
\end{proof}

\section{Weak associative laws with length 4 and with 4 variables.}

\medskip 
The only word of length 4 with 4 variables is: 
\medskip  

A: $\langle t, x, y, z\rangle$.

\medskip 
 \noindent  Ways in which a word of length 4 can be bracketed:\medskip  

1: $a \to (a \to (a \to a))$, 

2: $a \to ((a \to a) \to a)$, 

3: $(a \to a) \to (a \to a)$, 

4: $(a \to (a \to a)) \to a$, 

5: $((a \to a) \to a) \to a$. \\

Let $L$ be the set consisting of the following varieties $44\mathcal A12$, $44\mathcal A13 $, $ 44\mathcal A14 $, $ 44\mathcal A15 $, $ 44\mathcal A23 $, $ 44\mathcal A24 $, $ 44\mathcal A34 $, $ 44\mathcal A35 $ and $ 44\mathcal A45$.

\begin{Theorem} \label{Theorem_091116_01}
	If $\mathcal X \in L$ then $\mathcal X = 33\mathcal{A}12$.
\end{Theorem}

\begin{proof}

It is easy to see that $33\mathcal{A}12 \subseteq \mathcal X$.

		
		\medskip
		
		Let $\mathbf A \in \mathcal X$ and let 
		$a, b, c \in A$.
		
		To facilitate a uniform presentation (and to make the proof shorter), we introduce the following notation, where $u_0, x_0, y_0, z_0 \in A$:\\
		
		The notation	
		$$ \mathcal X/ u_0, x_0, y_0, z_0$$ denotes the following 
		statement: 
		
			``In the identity (X) that defines the variety $\mathcal X$, relative to $\mathcal S$, if we assign $u:= u_0, x:= x_0, y:= y_0, z:= z_0$ 
			(and simplify it using Lemma \ref{general_properties_equiv} (\ref{TXX}) and $x \approx x''$), 
			then $\mathbf{A} \models  x \to (y \to z) \approx (x \to y) \to z$''. \\  
		
\noindent Hence, we consider the varieties associated with the following statements: 
		
		
			




		

	\begin{multicols}{3}
		
\begin{enumerate}
\item ${\rm 44}\mathcal A\rm{12/ 0', a, b, c;}$
\item $ {\rm 44}\mathcal {\rm A13/ a, b, 0', c};$
\item ${\rm 44}\mathcal A\rm{14/ 0', a, b, c;}$
\item $ {\rm 44}\mathcal {\rm A15/ 0', a,b, c};$
\item $ {\rm 44}\mathcal A\rm{23/ 0', a, b, c;}$
\item ${\rm 44}\mathcal {\rm A24/ a, 0', b, c};$
\item ${\rm 34}\mathcal A\rm{12/ 0', a, b, c;}$
\item ${\rm 44}\mathcal {\rm A35/ 0', a, b, c};$
\item ${\rm 44}\mathcal A\rm{45/ a, b, c, 0.}$
\end{enumerate}
		
\end{multicols}

		It is routine to verify that each of the above statements is true, from which it follows that, 
for $\mathcal X \in L$, $\mathcal X \in 33\mathcal{A}12$.
\end{proof}

\begin{Lemma} \label{lemma_111116_01}
	If $\mathbf A \in 31\mathcal A12 \cup 43\mathcal A23 \cup 44\mathcal A25$ then  $\mathbf A \models 0 \to x \approx x$.
\end{Lemma}

\begin{proof}
	First, we will show that if $\mathbf A \in 31\mathcal A12 \cup 43\mathcal A23 \cup 44\mathcal A25$, then  $\mathbf A \models 0' \approx 0$.

Assume that $\mathbf A \in 31\mathcal A12$. Then,  by Lemma \ref{general_properties_equiv} (\ref{LeftImplicationwithtilde})  and (31A12), $0' = 0 \to 0' = 0 \to (0 \to 0) = (0 \to 0) \to 0 = 0' \to 0 = 0$.\\
If $\mathbf A \in 43\mathcal A23$,

$0 \overset{\ref{general_properties_equiv} (\ref{TXX})}{ =}  0' \to 0 
\overset{\ref{general_properties_equiv} (\ref{TXX})}{ =}  0' \to (0' \to 0) 
 =  (0 \to 0) \to (0' \to 0) 
\overset{(43A23)}{ =} 0 \to ((0 \to 0') \to 0) 
\overset{\ref{general_properties_equiv} (\ref{LeftImplicationwithtilde}) }{ =}  0 \to (0' \to 0) 
\overset{\ref{general_properties_equiv} (\ref{TXX})} { =}  0 \to 0 
 =  0'.$ 

If $\mathbf A \in 44\mathcal A25$,

$0'  =  0 \to 0 
\overset{\ref{general_properties_equiv} (\ref{TXX})}{ =}  0 \to (0' \to 0) 
 =  0 \to ((0 \to 0') \to 0) 
\overset{(44A25)}{ =}  ((0 \to 0) \to 0') \to 0 
 =  (0' \to 0') \to 0 
\overset{ \ref{general_properties_equiv} (\ref{TXX})}{ =}  0' \to 0 
\overset{\ref{general_properties_equiv} (\ref{TXX})}{ =} 0.$ 

Therefore, $$\mathbf A \models 0 \approx 0'.$$

 Hence, for $a \in A$, $0 \to a = 0' \to a = a$ by Lemma \ref{general_properties_equiv} (\ref{TXX}).
\end{proof}

\begin{Lemma} \label{lemma_111116_02}
	If $\mathbf A \models 0 \to x \approx x$ then  $\mathbf A \in 31\mathcal A12 \cap 43\mathcal A23 \cap 44\mathcal A25$.
\end{Lemma}

\begin{proof}
	Let $a,b,c,d \in A$.  Since  
$a \to (a \to a) \overset{\ref{general_properties2} (\ref{031114_04})}{ =}  a \to a 
\overset{hyp}{=}  (0 \to a) \to a 
\overset{\ref{general_properties2} (\ref{291014_10})}{=}  (a \to a) \to a,$ 
we have $\mathbf A \models (31A12)$.
Observe that, by Lemma \ref{lemma_5terms}  we have that 
$\mathbf A \models t_3(x,y,z,t) \approx t_5(x,y,z,t)$ and $\mathbf A \models t_3(x,y,z,t) \approx t_4(x,y,z,t)$. Hence,

$(a \to a) \to (b \to c)  =  t_3(a,a,b,c) 
 =  t_5(a,a,b,c) 
 =  (a \to b) \to (a \to c) 
\overset{\ref{properties_of_I20_MC} (\ref{310516_01})}{ =}  a \to ((a \to b) \to c)$ 
and consequently, the algebra $\mathbf A$ satisfies (43A23).  Also, the algebra satisfies the identity (44A25), since

$((d \to a) \to b) \to c =  t_4(b,a,d,c) 
 = t_3(b,a,d,c) 
 =  (a \to b) \to (d \to c) 
 \overset{\ref{properties_of_I20_MC} (\ref{310516_01})} {=}  d \to ((a \to b) \to c)$,  
completing the proof.
\end{proof}

	


\begin{Lemma} \label{lemma_061216_02}
	 $31\mathcal A12 = 43\mathcal A23 = 44\mathcal A25$ \label{061216_03}
\end{Lemma}





\begin{proof} 
	The proof is immediate from  Lemma \ref{lemma_111116_01} and Lemma \ref{lemma_111116_02}.
\end{proof}

\section{MAIN THEOREM} \label{section_main_theorem}

The purpose of this section is to prove our main theorem.  The following proposition will be needed to complete the proof of the main theorem.

In the following proposition we combine all the results obtained so far both in this paper and in 
\cite{cornejo2016BolMoufang}.  In the latter, each of the 60 varieties defined by the weak associative laws of lenth 4 with 3 variables was shown to be equal to one of the three varieties:  $43\mathcal A12$, $43\mathcal A23$, and $43\mathcal F25$.  Therefore, we will need only consider these three varieties with the remaining 95 varieties in the following result.


\begin{Proposition} \label{prop_to_Main_Theo}
Each of the 155 weak associative subvarieties of 
$\mathcal{S}$ 
is equal to one of the follwing varieties:
				$$\mathcal{SL},\  43\mathcal A12,\  43\mathcal A23,\  42\mathcal A12,\  43\mathcal F25\   \mbox{ and }  \mathcal{S}.$$
\end{Proposition}\label{New_Prop}

\begin{proof} We have

\begin{enumerate}

\item  $\mathcal{SL}  
\overset{\ref{theorem_131216_01}}{=} \mathcal{LALT}
\overset{\ref{theorem_131216_01}}{=} \mathcal{FLEX}
\overset{\ref{theorem_131216_01}}{=}\mathcal{RALT}   
\overset{\ref{Old_Lemma_10.3(3)}}{=} 33\mathcal A12 
\overset{\ref{theorem_42X_equalTo_SL}}{=} 42\mathcal A14  
\overset{\ref{theorem_42X_equalTo_SL}}{=} 42\mathcal B14 
\overset{\ref{theorem_42X_equalTo_SL}}{=} 42\mathcal B24 
\overset{\ref{theorem_42X_equalTo_SL}}{=} 42\mathcal B34
\overset{\ref{theorem_42X_equalTo_SL}}{=} 42\mathcal B45
\overset{\ref{theorem_42X_equalTo_SL}}{=} 42\mathcal C14
\overset{\ref{theorem_42X_equalTo_SL}}{=} 42\mathcal C24
\overset{\ref{theorem_42X_equalTo_SL}}{=} 42\mathcal C34
\overset{\ref{theorem_42X_equalTo_SL}}{=}42\mathcal C45
\overset{\ref{theorem_42X_equalTo_SL}}{=} 42\mathcal D14
\overset{\ref{theorem_42X_equalTo_SL}}{=}42\mathcal D34
\overset{\ref{theorem_42X_equalTo_SL}}{=} 42\mathcal E12
\overset{\ref{theorem_42X_equalTo_SL}}{=} 42\mathcal E14
\overset{\ref{theorem_42X_equalTo_SL}}{=} 42\mathcal E23
\overset{\ref{theorem_42X_equalTo_SL}}{=} 42\mathcal E24 
\overset{\ref{theorem_42X_equalTo_SL}}{=} 42\mathcal E34
\overset{\ref{theorem_42X_equalTo_SL}}{=} 42\mathcal E35
\overset{\ref{theorem_42X_equalTo_SL}}{=}42\mathcal E45
\overset{\ref{theorem_42X_equalTo_SL}}{=} 42\mathcal F15
\overset{\ref{theorem_42X_equalTo_SL}}{=} 42\mathcal F23
\overset{\ref{theorem_42X_equalTo_SL}}{=} 42\mathcal F24
\overset{\ref{theorem_42X_equalTo_SL}}{=} 42\mathcal F45
\overset{\ref{theorem_42X_equalTo_SL}}{=} 42\mathcal G23
\overset{\ref{theorem_42X_equalTo_SL}}{=} 42\mathcal G24
\overset{\ref{theorem_42X_equalTo_SL}}{=} 42\mathcal G35 
\overset{\ref{theorem_42X_equalTo_SL}}{=} 42\mathcal G45
\overset{\ref{lemma_081116_02} (\ref{231116_07})}{=} 42\mathcal E15 
\overset{\ref{lemma_081116_03}}{=} 42\mathcal G14
\overset{\ref{lemma_081116_03}}{=} 42\mathcal F13
\overset{\ref{lemma_081116_03}}{=} 42\mathcal F14
\overset{\ref{lemma_081116_03}}{=} 42\mathcal F34  
\overset{\ref{Theorem_091116_01}}{=} 44\mathcal A12
\overset{\ref{Theorem_091116_01}}{=} 44\mathcal A13
\overset{\ref{Theorem_091116_01}}{=} 44\mathcal A14
\overset{\ref{Theorem_091116_01}}{=} 44\mathcal A15
\overset{\ref{Theorem_091116_01}}{=} 44\mathcal A23
\overset{\ref{Theorem_091116_01}}{=} 44\mathcal A24
\overset{\ref{Theorem_091116_01}}{=} 44\mathcal A34
\overset{\ref{Theorem_091116_01}}{=} 44\mathcal A35
\overset{\ref{Theorem_091116_01}}{=} 44\mathcal A45.$   


\item    \label{131216_01}
$43\mathcal A12
\overset{\ref{lemma_061216_01}}{=} 42\mathcal C12 
\overset{\ref{theo_081116_02}}{=} 42\mathcal C13 
\overset{\ref{theo_081116_02}}{=} 42\mathcal C15 
\overset{\ref{theo_081116_02}}{=}42\mathcal D12 
\overset{\ref{theo_081116_02}}{=} 42\mathcal F35
\overset{\ref{lemma_081116_02} (\ref{231116_07})}{=}42\mathcal D15
\overset{\ref{lemma_081116_03}}{=} 42\mathcal E13
\overset{\ref{lemma_081116_03}}{=} 42\mathcal F12
\overset{\ref{lemma_081116_03}}{=} 42\mathcal G13.$  

\item    $43\mathcal A23 
\overset{\ref{lemma_061216_02}}{=}31\mathcal A12 
\overset{\ref{Old_Lemma 10.3(1)}}{=} 41\mathcal A14 
\overset{\ref{lemma_131216_01} (\ref{281016_06})}{=} 41\mathcal A34
\overset{\ref{lemma_131216_01} (\ref{281016_06})}{=} 41\mathcal A45 
\overset{\ref{lemma_131216_01} (\ref{281016_02})}{=} 41\mathcal A24 
\overset{\ref{lemma_061216_02}}{=} 44\mathcal A25
\overset{\ref{Old_Lemma_10.3(2)}}{=} 42\mathcal A24
\overset{\ref{lemma_081116_02} (\ref{231116_06})}{=}42\mathcal A34 
\overset{\ref{theorem_081116_01}}{=} 42\mathcal A45 
\overset{\ref{theorem_081116_01}}{=} 42\mathcal B13
\overset{\ref{theorem_081116_01}}{=}  42\mathcal B15 
\overset{\ref{theorem_081116_01}}{=} 42\mathcal B23
\overset{\ref{theorem_081116_01}}{=} 42\mathcal B25
\overset{\ref{theorem_081116_01}}{=} 42\mathcal D24
\overset{\ref{theorem_081116_01}}{=} 42\mathcal D45
\overset{\ref{theorem_081116_01}}{=} 42\mathcal F25
\overset{\ref{theorem_081116_01}}{=} 42\mathcal G15 
\overset{\ref{lemma_081116_03}}{=} 42\mathcal G34   
\overset{\ref{lemma_081116_03}}{=} 42\mathcal G12.$ 


\item    $42\mathcal A12 
\overset{\ref{lemma_081116_02} (\ref{231116_06})}{=} 42\mathcal A13 
\overset{\ref{lemma_081116_02} (\ref{231116_06})}{=} 42\mathcal A15 
\overset{\ref{lemma_081116_06}}{=} 42\mathcal D23
\overset{\ref{lemma_081116_06}}{=} 42\mathcal D35.$   

\item    $43\mathcal F25
\overset{\ref{Ola_Lemma_10.3(5)}}{=}42\mathcal B35 
\overset{\ref{theo_081116_01}}{=}  42\mathcal B12 
\overset{\ref{theo_081116_01}}{=} 42\mathcal C23 
\overset{\ref{theo_081116_01}}{=} 42\mathcal C25 
\overset{\ref{theo_081116_01}}{=} 42\mathcal C35
\overset{\ref{theo_081116_01}}{=} 42\mathcal G25
\overset{\ref{lemma_081116_03}}{=}42\mathcal D13.$  

\item    $\mathcal S 
\overset{\ref{lemma_131216_01} (\ref{281016_04})}{=} 41\mathcal A23 
\overset{\ref{lemma_131216_01} (\ref{281016_05})}{=} 41\mathcal A13 
\overset{\ref{lemma_131216_01} (\ref{281016_05})}{=} 41\mathcal A15 
\overset{\ref{lemma_131216_01} (\ref{281016_05})}{=}41\mathcal A35 
\overset{\ref{lemma_131216_01} (\ref{281016_03})}{=} 41\mathcal A25 
\overset{\ref{lemma_131216_01} (\ref{281016_01})}{=} 41\mathcal A12
\overset{\ref{lemma_081116_02} (\ref{311016_03})}{=} 42\mathcal A23 
\overset{\ref{lemma_081116_02} (\ref{311016_03})}{=} 42\mathcal A25
\overset{\ref{lemma_081116_02} (\ref{311016_03})}{=} 42\mathcal A35 
\overset{\ref{lemma_081116_02} (\ref{311016_03})}{=} 42\mathcal D25  
\overset{\ref{lemma_081116_02} (\ref{311016_03})}{=} 42\mathcal E25.$  
\end{enumerate}
\end{proof}



We are now ready to present the main theorem of this paper.

%

\begin{Theorem} \label{theo_5_dist_var}
	\begin{thlist}
		\item[a]	The following are the 6 varieties defined, relative to $\mathcal S$, by the 155 weak associative laws of length $m \leq 4$ that are distinct from each other:
		\begin{center}
			$\mathcal{SL}$, $43\mathcal A12$, $43\mathcal A23$, $42\mathcal A12$, $43\mathcal F25$ and $\mathcal{S}$. 
		\end{center}
		\item[b]  They satisfy the following relationships: 
		\begin{enumerate}
			\item $\mathcal{SL} \subset 43\mathcal A23 \subset 43\mathcal F25 \subset \mathcal S$,
			\item $\mathcal{SL} \subset 43\mathcal A12 \subset 42\mathcal A12 \subset \mathcal S$,
			\item $\mathcal{BA} \subset 43\mathcal A12 \subset 43\mathcal F25$,
			\item $43\mathcal A12 \not\subseteq 43\mathcal A23$ and $43\mathcal A23 \not\subseteq 43\mathcal A12$,
			\item $42\mathcal A12 \not\subseteq 43F25$ and $43\mathcal F25 \not\subseteq 42\mathcal A12$.
		\end{enumerate}
	\end{thlist}	
\end{Theorem}

\begin{proof} We first prove (b):
\begin{enumerate}
	\item Follows from Theorem \ref{theo_4_dist_var}. 
	
	\item The statement $\mathcal{SL} \subset 43\mathcal A12$ follows from Theorem \ref{theo_4_dist_var}. Then it remains to check that $43\mathcal A12 \subset 42\mathcal A12 \subset \mathcal S$. 
	
	Let $\mathbf A \in 43\mathcal A12$ and $a,b \in A$. By Proposition \ref{prop_to_Main_Theo} (\ref{131216_01}), $\mathbf A \in 42\mathcal C12$.
	Observe that 
	\begin{equation}  \label{141116_04}
	42\mathcal C12 \models 0 \to x \approx x \to x
	\end{equation}
	 by Lemma \ref{lema_1_identityC}. Then, using Lemma \ref{lemma_081116_A} (\ref{231116_03}), we have $42\mathcal C12 \models 0 \to (x \to x) \approx x \to x$. Hence, by Lemma \ref{lemma_071116_A} (\ref{081116_02}),
	 \begin{equation}  \label{141116_03}
	 \mathbf A \models (x \to x) \to y' \approx ((x \to x) \to y)'
	 \end{equation}
	 Hence
	 \noindent $a \to ((a \to a) \to b)$
	 $\overset{   \ref{properties_of_I20_MC} (\ref{310516_01}) 
	 }{=}  (a \to a) \to (a \to b)  $
	 $\overset{   \ref{properties_of_I20_MC} (\ref{310516_02}) 
	 }{=}  (a \to a) \to (b' \to a') $
	 $\overset{ x \approx x'' 
	 }{=}  ((a \to a) \to (b' \to a'))'' $
	 $\overset{  (\ref{141116_03}) 
	 }{=}  ((a \to a) \to (b' \to a')')' $
	 $\overset{   \ref{properties_of_I20_MC} (\ref{310516_02}) 
	 }{=}  ((a' \to a') \to (b' \to a')')' $
	 $\overset{  (\ref{141116_04}) 
	 }{=}  ((0 \to a') \to (b' \to a')')' $
	 $\overset{   \ref{properties_of_I20_MC} (\ref{310516_02}) 
	 }{=}  ((a \to 0') \to (b' \to a')')' $
	 $\overset{  (I) 
	 }{=}  (0' \to b') \to a' $
	 $\overset{   \ref{general_properties_equiv} (\ref{TXX}) 
	 }{=}  b' \to a' $
	 $\overset{   \ref{properties_of_I20_MC} (\ref{310516_02}) 
	 }{=}  a \to b $
	 $\overset{   \ref{general_properties2} (\ref{031114_04}) 
	 }{=}  a \to (a \to (a \to b)) $.
	 Thus, $\mathbf A \in 42\mathcal A12$.
	 The following examples show, respectively, that $42\mathcal A12 \not= 43\mathcal A12$ and $42\mathcal A12 \not= \mathcal S$. \\
	 
	 \begin{minipage}{0.5 \textwidth}
	 \begin{tabular}{r|rrr}
	 	$\to$: & 0 & 1 & 2\\
	 	\hline
	 	0 & 2 & 2 & 2 \\
	 	1 & 1 & 1 & 2 \\
	 	2 & 0 & 1 & 2
	 \end{tabular}
 \end{minipage}
 \begin{minipage}{0.5 \textwidth}
	 \begin{tabular}{r|rrrr}
	 	$\to$: & 0 & 1 & 2 & 3\\
	 	\hline
	 	0 & 0 & 1 & 2 & 3 \\
	 	1 & 2 & 3 & 2 & 3 \\
	 	2 & 1 & 1 & 3 & 3 \\
	 	3 & 3 & 3 & 3 & 3
	 \end{tabular} 
\end{minipage}
	 
\item Follows from Theorem \ref{theo_4_dist_var}. 

\item  
The algebra $\mathbf{2_b}$ shows that $43\mathcal A12 \not\subseteq 43\mathcal A23$, whilst the following algebra shows that 
$43\mathcal A23 \not\subseteq 43\mathcal A12$.\\




\begin{tabular}{r|rrrr}
	$\to$: & 0 & 1 & 2 & 3\\
	\hline
	0 & 0 & 1 & 2 & 3 \\
	1 & 2 & 3 & 2 & 3 \\
	2 & 1 & 1 & 3 & 3 \\
	3 & 3 & 3 & 3 & 3
\end{tabular}

\item The following algebras show that $42\mathcal A12 \not\subseteq 43\mathcal F25$ and $43\mathcal F25 \not\subseteq 42\mathcal A12$, respectively.\\

\begin{minipage}{0.5 \textwidth}
\begin{tabular}{r|rrr}
	$\to$: & 0 & 1 & 2\\
	\hline
	0 & 2 & 2 & 2 \\
	1 & 1 & 1 & 2 \\
	2 & 0 & 1 & 2
\end{tabular}
\end{minipage}
\begin{minipage}{0.5 \textwidth}
\begin{tabular}{r|rrrr}
	$\to$: & 0 & 1 & 2 & 3\\
	\hline
	0 & 0 & 1 & 2 & 3 \\
	1 & 2 & 3 & 2 & 3 \\
	2 & 1 & 1 & 3 & 3 \\
	3 & 3 & 3 & 3 & 3
\end{tabular}
\end{minipage}


\end{enumerate}	
	The proof of the theorem is now complete since (a) is an immediate consequence of Proposition \ref{prop_to_Main_Theo} and (b).
\end{proof}

The Hasse diagram of the poset of the weak associative subvarieties of $\mathcal S$ of length $\leq 4$, together with the variety $\mathcal{BA}$:

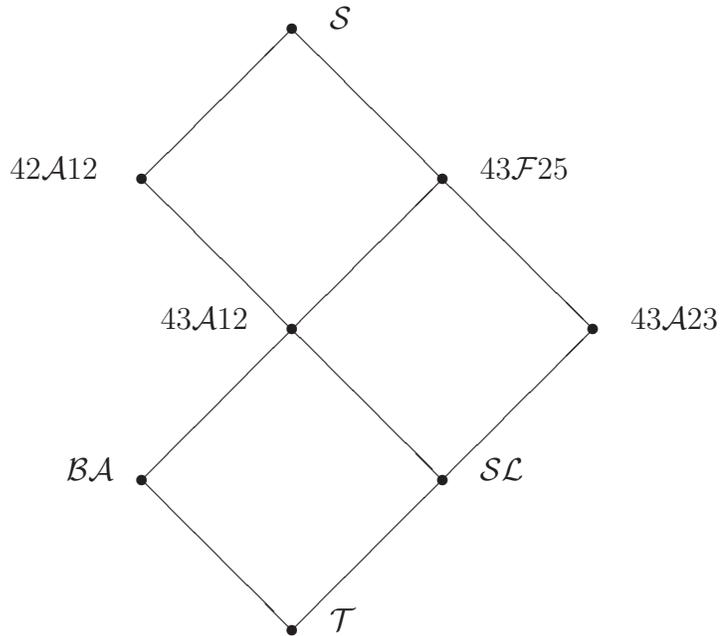
\setlength{\unitlength}{0.5cm}
\begin{picture}(28,22)(0,0)


\put(12,4){\circle*{0.3}}
\put(13,4){$\mathcal T$}
\put(8,8){\circle*{0.3}}
\put(6,8){$\mathcal{BA}$}
\put(16,8){\circle*{0.3}}
\put(17,8){$\mathcal{SL}$}
\put(12,12){\circle*{0.3}}
\put(8.5,12){$43 \mathcal A12$}
\put(20,12){\circle*{0.3}}
\put(21,12){$43\mathcal A23$}
\put(8,16){\circle*{0.3}}
\put(4.5,16){$42\mathcal A12$}
\put(16,16){\circle*{0.3}}
\put(17,16){$43\mathcal F25$}
\put(12,20){\circle*{0.3}}
\put(13,20){$\mathcal S$}


\put(12,4){\line(1,1){8}}
\put(8,8){\line(1,1){8}}
\put(8,16){\line(1,1){4}}


\put(12,20){\line(1,-1){8}}
\put(8,16){\line(1,-1){8}}
\put(8,8){\line(1,-1){4}}

\end{picture}

We conclude the paper with the remark that it would be of interest to investigate the weak associative identities, and in particular, the identities of Bol-Moufang type, relative to $\mathcal{I}$ and other (important) subvarieties of $\mathcal{I}$.\\

\ \\ \ \\
\noindent{\bf Acknowledgment:} 
The first author wants to thank the
institutional support of CONICET  (Consejo Nacional de Investigaciones Cient\'ificas y T\'ecnicas).   
\ \\

\noindent {\bf Compliance with Ethical Standards:}\\ 

\noindent {\bf Conflict of Interest:} The first author declares that he has no conflict of interest. The second author  declares that he has no conflict of interest. \\

\noindent {\bf Ethical approval:}
This article does not contain any studies with human participants or animals performed by any of the authors. \\

	\noindent {\bf Funding:}  
	The work of Juan M. Cornejo was supported by CONICET (Consejo Nacional de Investigaciones Cientificas y Tecnicas) and Universidad Nacional del Sur. 
	Hanamantagouda P. Sankappanavar did not receive any specific grant from funding agencies in the public, commercial, or not-for-profit sectors.

\ \\

\noindent {\sc Juan M. Cornejo}\\
Departamento de Matem\'atica\\
Universidad Nacional del Sur\\
Alem 1253, Bah\'ia Blanca, Argentina\\
INMABB - CONICET\\

\noindent jmcornejo@uns.edu.ar

\vskip 1.4cm

\noindent {\sc Hanamantagouda P. Sankappanavar}\\
Department of Mathematics\\
State University of New York\\
New Paltz, New York 12561\\
U.S.A.\\

\noindent sankapph@newpaltz.edu

\end{document}